\documentclass[a4paper,11pt]{amsart}
\usepackage{amssymb,amscd} 

\title{A survey on Almost Difference Sets}
\theoremstyle{plain}

\newtheorem{thm}{Theorem}
\newtheorem{prop}[thm]{Proposition}
\newtheorem{lem}[thm]{Lemma}
\newtheorem{cor}[thm]{Corollary}

\newtheorem{rem}[thm]{Remark}
\newtheorem{exa}[thm]{Example}

\newcommand{\mc}{\mathcal}

\newcommand{\hs}{\hspace{1mm}}

\date{\today}     

\begin{document} \pagenumbering{arabic} \setcounter{page}{1}
\author[K.~Nowak]{Kathleen Nowak}
\thanks{Department of Mathematics, Iowa State
University, Ames, IA 50011, USA (knowak@iastate.edu). }

\address{Department of Mathematics, Iowa State University,
Ames, Iowa, 50011, U. S. A.} \email[K.~Nowak]{knowak@iastate.edu}

\maketitle
\thanks

\begin{abstract} Let $G$ be an additive group of order $v$. A $k$-element subset $D$ of $G$ is called a $(v, k, \lambda, t)$-almost difference set if the expressions $gh^{-1}$, for $g$ and $h$ in $D$, represent $t$ of the non-identity elements in $G$ exactly $\lambda$ times and every other non-identity element $\lambda+1$ times.  Almost difference sets are highly sought after as they can be used to produce functions with optimal nonlinearity, cyclic codes, and sequences with three-level autocorrelation. This paper reviews the recent work that has been done on almost difference sets and related topics. In this survey, we try to communicate the known existence and nonexistence results concerning almost difference sets. Further, we establish the link between certain almost difference sets and binary sequences with three-level autocorrelation. Lastly, we provide a thorough treatment of the tools currently being used to solve this problem. In particular, we review many of the construction methods being used to date, providing illustrative proofs and many examples. \\

{\it Keywords:} cyclotomic classes, group ring, difference set, partial difference set. \\

\end{abstract}

\newpage
\section{Introduction}

Almost difference sets are interesting combinatorial objects which have a wide range of applications. They are used in many areas of engineering including cryptography, coding theory, and CDMA communications. In cryptography, they are employed to construct functions with optimal nonlinearity (\cite{CDR}, \cite{Di94}). Additionally, in coding theory, they can be used to construct cyclic codes \cite{Di13}. Finally, they have garnered a lot of attention from CDMA communication since some cyclic almost difference sets yield sequences with optimal three-level autocorrelation (\cite{AD}, \cite{DHL}, \cite{DH}). 

Let $G$ be an additive group of order $v$ with identity $0$. A $k$-element subset $D \subset G$ is called a $(v, k, \lambda, t)$-\textit{almost difference set} (ADS) if the multiset
$$
\mathcal{D} = ( \hs gh^{-1} \hs | \hs g, h  \in D)
$$
contains $t$ elements of $G \setminus \{0\}$ with multiplicity $\lambda$ and the remaining nonidentity elements with multiplicity $\lambda+1$.  An almost difference set $D$ is called \textit{abelian} or \textit{cyclic} if the group $G$ is abelian or cyclic respectively.  We now give an equivalent definition which may be preferable depending on the application. A $k$-subset $D \subset G$ is said to be a $(v, k, \lambda, t)$-\textit{almost difference set} (ADS) if the difference function $d_D(x) : = |(D+x) \cap D|$ takes on the value $\lambda$ altogether $t$ times and $\lambda +1$ altogether $v-t-1$ times as $x$ ranges over $G\setminus \{0\}$. That is,
$$
d_D(x) = |(D+x) \cap D| = \lambda \text{ or } \lambda+1
$$
for each $x \in G \setminus \{0\}$.

Different definitions of an almost difference set were independently developed by Davis and Ding in the early 1990s. For an additive group $G$ with subgroup $H$, we call a $k$-subset $D \subset G$  a $(|G|, |H|, k, \lambda_1, \lambda_2)$-\textit{divisible difference set} (DDS) if the difference function $d_D(x)$ defined above takes on the value $\lambda_1$ for each nonzero $x \in H$ and takes on the value $\lambda_2$ for each nonzero $x \in G\setminus H$. That is, for each $x \in G \setminus \{0\}$,
$$
d_D(x) =
\begin{cases}
\lambda_1 , & \text{for }x \in H \\
\lambda_2 , & \text{for } x \in G \setminus H
\end{cases}
$$
Davis \cite{Da} called a divisible difference set an almost difference set in the case where $|\lambda_1 - \lambda_2| = 1$. He highlighted this case as it relates to symmetric difference sets. In a different light, Ding originally restricted the current definition to the case when $t = \frac{v-1}{2}$ for constructing cryptographic functions with optimal nonlinearity. The two notions were generalized to the current definition by Ding, Helleseth, and Martinsen \cite{DH} with the aim of constructing binary sequences with three-level autocorrelation. Note that the two original ideas are now special cases of this modern definition.

 \indent Observe that when $t = 0$ or $v-1$, $D$ is an $(v, k, \lambda)$-\textit{difference set} (DS). Difference sets have been extensively studied. Their use to construct symmetric block designs dates back to R.C. Bose and his 1939 paper \cite{Bose}. However, examples such as the Paley difference set appeared even earlier. For more information on difference sets, we refer the reader to Beth-Jungnickel-Lenz \cite{BJ}. From henceforth, we will restrict our attention to when $0 < t < v-1$.

The rest of this paper consists of the following nine sections:
(2) Motivation: almost difference sets and binary sequences with three-level autocorrelation; In this section we motivate our study by establishing a link between almost difference sets and binary sequences with three-level autocorrelation. (3) Difference sets and group rings; Here we introduce a tool for working with different types of difference sets (families). (4) Some necessary conditions for the existence of almost difference sets; In this section, we summarize the known parameter restrictions for almost difference sets.
(5) Difference sets, partial difference sets, and almost difference sets; Here we highlight the interplay between difference sets, partial difference sets, and almost difference sets. (6) Almost difference sets from cyclotomic classes of finite fields; (7) Almost difference sets from functions; (8) Almost difference sets from binary sequences with three-level autocorrelation; (9) Almost difference sets from direct product constructions; In sections 6-9, we summarize the known constructions. (10) Concluding remarks and open problems.

\section{Motivation: Almost Difference Sets and Binary Sequences with Three-Level Autocorrelation}

In this section, we motivate our studies by divulging how almost difference sets are related to binary sequences with good autocorrelation properties. Throughout this section $\{s(t)\}$ will denote a periodic binary sequence with period $n$.

\indent For a subset $D$ of the ring $\mathbb{Z}_n$, we define the \textit{characteristic sequence} $\{s(t)\}$ of $D$ by
 $$
s(t) =
\begin{cases}
1, & \text{if } t (\text{mod }n) \in D \\
0, & \text{otherwise }
\end{cases}
$$
Note that defining $\{s(t)\}$ in this way yields a periodic sequence with period $n$. Now going in the opposite direction, given a periodic binary sequence, $\{s(t)\}$, of period $n$, we call $D = \{t \in \mathbb{Z}_{n} \hs | \hs s(t) = 1 \}$ the \textit{characteristic set} or \textit{support} of $\{s(t)\}$.

\indent Next, for two periodic binary sequences, $\{s(t)\}$ and $\{u(t)\}$, of period $n$, we define their \textit{periodic cross correlation function} as follows:
$$
C_{s,u} (w) = \sum_{t \in \mathbb{Z}_n} (-1)^{s(t+w)-u(t)}
$$

As a special case, we can consider when $u(t) = s(t)$. In this case, we define the \textit{autocorrelation function} of $\{s(t)\}$ as
$$
C_s(w) = \sum_{t \in \mathbb{Z}_n} (-1)^{s(t+w)-s(t)}
$$

\indent That is, for a shift $w$, $C_s(w)$ is the difference between the number of positions that match and and the number of positions that differ over one period between the shifted sequence $\{s(t+w)\}$ and the original sequence $\{s(t)\}$.

\indent Pseudorandom sequences have applications in simulation, software testing, global positioning systems, ranging systems, code-division multiple-access systems, radar systems, broad-spectrum communication systems, and stream ciphers \cite{DHL}. Further, many of these applications require binary sequences with good autocorrelation properties \cite{CDR}. Specifically, the goal is to obtain sequences whose autocorrelation values are as small in absolute value as possible while minimizing the total number of distinct autocorrelation values. To characterize the number of distinct autocorrelation values, we say that $\{s(t)\}$ has \textit{$k$-level autocorrelation} if $C_s(w)$ takes on $k$ distinct values. Further, $\{s(t)\}$ is said to have \textit{optimal autocorrelation} if for all $w \not \equiv 0 (\text{mod }n)$,
$$
C_s(w) \in
\begin{cases}
\{-1\}, & \text{if } n \equiv 3(\text{mod }4) \\
\{1,-3\} & \text{if } n \equiv 1(\text{mod }4) \\
\{-2,2\}, & \text{if } n \equiv 2(\text{mod }4) \\
\{0,4\} \text{ or } \{0, -4\}, & \text{if } n \equiv 0(\text{mod }4)
\end{cases}
$$
and $\{s(t)\}$ is said to have \textit{ideal autocorrelation} if $n \equiv 3 (\text{mod }4)$ and $C_s(w) = -1$ for all $w \not \equiv 0(\text{mod }n)$.  Lastly, as a secondary goal it is often beneficial to have the number of ones in a periodic segment as close to $\frac{n}{2}$ as possible. Such sequences are said to have optimum balance between $0$'s and $1$'s.

\indent The following lemma provides the relation between autocorrelation values and (almost) difference sets.

\begin{lem} \cite{CDR} \label{lem2}
Let $\{s(t)\}$ be a periodic binary sequence with period $n$ and support $D$. Then
$$
C_s(w) = n - 4(k-d_D(w))
$$
where $k = |D|$.
\end{lem}

\begin{proof}
Fix a shift $w \in \mathbb{Z}_n$. Then there are four possible combinations for $s(t+w)$ and $s(t)$ which we summarize below.
\[\begin{array}{| c | c | c | c |}
\hline
s(t+w) & s(t) & t & (-1)^{s(t+w)-s(t)}  \\
\hline
0 & 0 & t \in (D-w)^* \cap D^* & 1  \\
0& 1 & t \in (D-w)^* \cap D & -1   \\
1 & 0 & t \in (D-w) \cap D^* & -1 \\
1 & 1 & t \in (D-w) \cap D &  1  \\
\hline
\end{array}\]
Thus,
\begin{eqnarray*}
C_s(w) &=& \sum_{t \in \mathbb{Z}_n}(-1)^{s(t+w)-s(t)} \\
&=& |((D-w)^* \cap D^*) \cup( (D-w) \cap D)|-| ((D-w)^* \cap D) \cup ((D-w) \cap D^*) | \\
&=& (|(D-w)^* \cap D^*| + |(D-w) \cap D| )- (|(D-w)^* \cap D| + |(D-w) \cap D^*|) \\
&=& ((n-2k+d_D(w)) +d_D(w)) - ((k-d_D(w))+(k-d_D(w))) \\
&=& n-4(k-d_D(w))
\end{eqnarray*}

\end{proof}

Hence, $\{s(t)\}$ has two-level autocorrelation if and only if $d_D(w)$ is two valued. That is to say, $D$ is a difference set in $\mathbb{Z}_n$. Unfortunately, in most cases $\mathbb{Z}_n$ does not contain a difference set. Thus, we are led to search for sequences with three-level correlation. By the lemma above, these are precisely the sequences for which $d_D(w)$ is three valued. Our final theorem for this section serves as a bridge between cyclic almost difference sets and some binary sequences with three-level autocorrelation.

\begin{thm}
Let $\{s(t)\}$ be a binary sequence with period $n$ and let $D$ be its support with $|D| = k$.  Then $D$ is an $(n, k, \lambda, t)$-almost difference set in $\mathbb{Z}_n$ if and only if the autocorrelation function $C_s(w)$ takes on the value $n-4(k-\lambda)$ altogether $t$ times and $n-4(k-\lambda-1)$ altogether $n-1-t$ times as $w$ ranges over all the nonzero elements of $\mathbb{Z}_n$.
\end{thm}

\begin{proof}
This follows from Lemma \ref{lem2} and the definition of an almost difference set.
\end{proof}

That is, starting with an almost difference set $D$ in $\mathbb{Z}_n$, the characteristic sequence for $D$ has three-level autocorrelation. On the flip side, given a sequence $\{s(t)\}$ with period $n$ and characteristic set $D$, if $\{s(t)\}$ has three-level autocorrelation and correlation values
$$
C_s(w) =
\begin{cases}
n, & \text{for } w = 0 \\
n-4(k-\lambda), & \text{for } t \text{ nonzero elements in } \mathbb{Z}_n \\
n-4(k-\lambda-1), &\text{for the remaining } n-t-1 \text{ nonzero elements in } \mathbb{Z}_n
\end{cases}
$$
then $D$ in an $(n,k,\lambda, t)$-almost difference set in $\mathbb{Z}_n$.

\section{Difference Sets and Group Rings}

An ADS is characterized by the multiplicities of certain group elements coming as differences taken within a set. Therefore, we start by providing an efficient tool for working with differences within a group.

\indent Let $G$ be an additive abelian group and let $\mathbb{Z}$ denote the ring of integers. Then we define the group ring $\mathbb{Z}[G]$ to be the ring of formal sums

$$
\mathbb{Z}[G] = \left \{\sum_{g \in G}a_g X^{g} \hs | \hs a_g \in \mathbb{Z}\right \}
$$

where $X$ is an indeterminate variable.  The ring $\mathbb{Z}[G]$ has the operation of addition given by
  $$
  \sum_{g \in G}a_gX^g + \sum_{g \in G} b_gX^g = \sum_{g \in G}(a_g+b_g)X^g
  $$

  and the operation of multiplication defined by

  $$
  \left (\sum_{g \in G}a_gX^g \right ) \left (\sum_{g \in G} b_gX^g \right ) = \sum_{h \in G} \left (\sum_{g \in G}a_gb_{h-g}\right )X^h
  $$

The zero and unit of $\mathbb{Z}[G]$ are $\sum_{g \in G}0X^g := \underline{0}$ and $X^0 := \underline{1}$ respectively. If $S$ is a subset of $G$, we will identify $S$ with the group ring element $S(X) = \sum_{s \in S}X^s$. This sum is sometimes referred to as a \textit{simple quantity}.

With this new terminology, we are able to provide a more succinct characterization of an almost difference set.

\begin{prop}\label{prop4}
Let $G$ be an additive abelian group of order $v$ and let $D$ be a $k$-subset of $G$. Then $D$ is an $(v,k,\lambda, t)$-ADS in $G$ if and only if there exists a $t$-subset $S$ of $G$ such that 
$$
D(X)D(X^{-1}) = k\cdot \underline{1} + \lambda S(X) + (\lambda+1)(G(X)-S(X)-\underline{1})
$$
in $\mathbb{Z}[G]$. 
\end{prop}

Furthermore, we can also provide a group ring characterization for difference sets, partial difference sets, and divisible difference sets (\cite{BJ}, \cite{Jung}, \cite{Ma94}).

\begin{prop}\label{prop6}
Let $G$ be an additive abelian group of order $v$ and let $D$ be a $k$-subset of $G$. Then $D$ is a $(v,k,\lambda)$-difference set in $G$ if and only if in $\mathbb{Z}[G]$,
$$
D(X)D(X^{-1}) = k\cdot \underline{1} +  \lambda(G(X)-\underline{1}).
$$
\end{prop}

\begin{prop}\label{prop5}
Let $G$ be an additive abelian group of order $v$ and let $D$ be a $k$-subset of $G$. Then $D$ is a $(v,k,\lambda, \mu)$-partial difference set in $G$ if and only if in $\mathbb{Z}[G]$,
$$
D(X)D(X^{-1}) =
\begin{cases}
k\cdot \underline{1} + \lambda (D(X)-\underline{1}) + \mu(G(X)-D(X)), & \text{if } 1 \in D \\
k\cdot \underline{1} + \lambda D(X) + \mu(G(X)-D(X)-\underline{1}), & \text{if } 1 \notin D
\end{cases}
$$
\end{prop}

\begin{prop}
Let $G$ be an additive abelian group of order $v$, $H$ a subgroup of $G$ of order $m$, and $D$ a $k$-subset of $G$. Then $D$ is a $(v, m, k, \lambda_1, \lambda_2)$-divisible difference set in $G$ relative to $H$ if and only if in $\mathbb{Z}[G]$,
$$
D(X)D(X^{-1}) = k \cdot \underline{1} + \lambda_1 (H(X)-\underline{1}) + \lambda_2(G(X)-H(X)).
$$
\end{prop}

\section{Some Necessary Conditions for the Existence of Almost Difference Sets} \label{nec}

Similar to difference sets, there are many parameter sets for which no ADS exists. In this section, we provide some necessary conditions for the existence of almost difference sets in $\mathbb{Z}_v$.

Our first necessary condition comes from counting the nonzero differences, $(d_i-d_j \hs | \hs d_i, d_j \in D, \hs d_i \neq d_j)$, of an ADS $D$ in two ways.

\begin{thm}\label{thm13}
Let $G$ be a group of order $v$. For a $(v, k, \lambda, t)$-ADS $D \subset G$, we have the following necessary condition:
$$
k(k-1) = \lambda t + (\lambda+1)(v-1-t).
$$
\end{thm}

\noindent Analogous to difference sets, the next theorem allows us to restrict our attention to sets $D$ with $|D| \leq \frac{v}{2}$.

\begin{thm} \cite{AD} \label{comp} $D$ is a $(v, k, \lambda, t)$-ADS in an abelian group $(G, +)$ if and only if its complement $D^* = G \setminus D$ is a $(v, v-k, v-2k+\lambda, t)$-ADS.
\end{thm}

\noindent In the case where $G = \mathbb{Z}_v$, we also have the following property.

\begin{thm} \cite{AD} If $D$ is a $(v, k, \lambda, t)$-ADS in $\mathbb{Z}_v$, then for any $a \in \mathbb{Z}_v$ with $gcd(a,v) = 1$ and any $b \in \mathbb{Z}_v$, $aD+b$ is a $(v,k, \lambda, t)$-ADS as well.
\end{thm}

From henceforth we restrict our attention to almost difference sets in the group $\mathbb{Z}_v$. The results we summarize come from a paper titled ``A New Family of Almost Difference Sets and Some Necessary Conditions" by Yuan Zhang, Jian Guo Lei, and Shao Pu Zhang.

For their first set of necessary conditions, we introduce some notation which generalizes the group ring concept presented earlier. For a set $D = \{d_1, d_2, ..., d_k\}$, the polynomial
$$
D(X) = X^{d_1}+X^{d_2}+ \cdots +X^{d_k}
$$
\noindent is called the \textit{Hall polynomial}.

\noindent With this notation, we have the following characterization:
\begin{thm}
$D$ is a $(v, k, \lambda, t)$-ADS in $G = \mathbb{Z}_v$ if and only if, there exists a $t$-subset $S$ of $\mathbb{Z}_v$ such that
$$
D(X)D(X^{-1}) \equiv k\cdot X^{0} +(\lambda + 1)(G(X)-1) - S(X) (\text{mod } X^v-1)
$$
where $H(X) (\text{mod } X^v-1)$ is obtained by identifying $X^v$ with $1$ in $H(X)$.
\end{thm}

\begin{proof}
This is equivalent to the group ring characterization for $G = \mathbb{Z}_v$ in Proposition \ref{prop4}.
\end{proof}

The next set of conditions come from using the above theorem to compute $D(X)D(X^{-1})$ in two ways.

\begin{thm}\cite{ZLZ}
Suppose that $D = \{d_1, d_2, ..., d_k\}$ is a $(v, k, \lambda, t)$-ADS in $\mathbb{Z}_v$. For each $w \vert v$, let
$$
c_0 + c_1X+c_2X^2 + \cdots c_{w-1}X^{w-1} \equiv S(X) (\text{mod }X^{w}-1)
$$
where $S$ is the $t$-set of elements which appear $\lambda$ times as differences of $D$.
Then,
$$
c_0+c_1+ \cdots + c_{w-1} = t
$$
and the following set of equations:
$$
\begin{cases}
\sum_{i=0}^{w-1} b_i = k \\
\sum_{i = 0}^{w-1}b_i^2 = k-(\lambda+1)+(\lambda+1)\frac{v}{w}-c_0 \\
\sum_{i = 0}^{w-1} b_ib_{i-j} = (\lambda+1)\frac{v}{w} - c_j \hs 1 \leq j \leq w-1 \hs \hs (\text{where the subscripts of }b_{i-j} \text{ are modulo }w)
\end{cases}
$$
have a nonnegative integer solution set $(b_0, b_1, ..., b_{w-1})$.
\end{thm}

\begin{exa} {\rm In $\mathbb{Z}_{80}$, no $(80, 13, 1, 2)$-ADS exists.
\begin{proof}
For $w = 2$, it is easy to check that the equations
$$
\begin{cases}
 b_0+b_1 = 13 \\
 b_0^2+b_1^2 = 13-2+2\cdot 40-c_0 \\
 b_0b_1+b_1b_0 = 2 \cdot 40 - c_1
\end{cases}
$$
have no nonnegative integer solution set. Thus, by the previous theorem, no $(80, 13, 1, 2)$-ADS exists in $\mathbb{Z}_{80}$.
\end{proof}}
\end{exa}

\noindent As specific applications of the theorem above, we have the following two corollaries.

\begin{cor} \label{cor14} \cite{ZLZ} When $k$ is odd,
$\vspace{.5mm}$

\begin{enumerate}
\item No $(v, k, \lambda, 1)$-ADS exists in $\mathbb{Z}_v$ if $v \equiv 4 (\text{mod }8)$.
$\vspace{.5mm}$
\item No $(v, k, \lambda, 1)$-ADS exists in $\mathbb{Z}_v$ if:
\begin{itemize}
\item $v \equiv 2(\text{mod }8)$ and $\lambda \equiv 2(\text{mod }4)$; or
\item $v \equiv 6(\text{mod }8)$ and $\lambda \equiv 0(\text{mod }4)$
\end{itemize}
\end{enumerate}
\end{cor}

\begin{exa}{\rm
From the necessary condition in Theorem \ref{thm13}, if $t = 1$, then $2 \vert v$ and $2 \vert \lambda$.
For $\mathbb{Z}_v$ with $v \leq 200$, we list (up to complementation) the parameter sets with $t = 1$ and $k$ odd which satisfy the necessary condition in Theorem \ref{thm13}. Of these, we frame the ones which are then ruled out by Corollary \ref{cor14}.
\[\begin{array}{ccccc}
(2, 1, 0, 1),& (8, 3, 0, 1),& \boxed{(22, 5, 0, 1)},& (38, 11, 2, 1),& (40, 17, 6, 1),\\
 \boxed{(44, 7, 0, 1)},& \boxed{(50, 19, 6, 1)}, &(74, 9, 0, 1), & (92, 17, 2, 1),& (104, 47, 20, 1),\\
  (112, 11, 0, 1),& (134, 31, 6, 1), & \boxed{(140, 43, 12, 1)}, &(152, 33, 6, 1),& \boxed{(158, 13, 0, 1)},\\
   \boxed{(164, 59, 20, 1)},& \boxed{(170, 23, 2, 1)},& \boxed{(182, 49, 12, 1)},& (194, 85, 36, 1), &(200, 93, 42, 1).\\
\end{array}\]}
\end{exa}

\begin{cor} \label{cor16} \cite{ZLZ} When $k$ is odd,
$\vspace{.5mm}$

\begin{enumerate}
\item No $(v, k, \lambda, v-2)$-ADS exists in $\mathbb{Z}_v$ if $v \equiv 4 (\text{mod }8)$.
$\vspace{.5mm}$
\item No $(v, k, \lambda, v-2)$-ADS exists in $\mathbb{Z}_v$ if:
\begin{itemize}
\item $v \equiv 2(\text{mod }8)$ and $\lambda \equiv 1(\text{mod }4)$; or
\item $v \equiv 6(\text{mod }8)$ and $\lambda \equiv 3(\text{mod }4)$
\end{itemize}
\end{enumerate}
\end{cor}

\begin{exa}{ \rm
From the necessary condition in Theorem \ref{thm13}, if $t = v-2$, then $2 \vert v$ and $2 \vert \lambda+1$.
For $\mathbb{Z}_v$ with $v \leq 200$, we list (up to complementation) the parameter sets with $t = v-2$ and $k$ odd which satisfy the necessary condition in Theorem \ref{thm13}. Of these, we frame the ones which are then ruled out by Corollary \ref{cor16}.
\[\begin{array}{ccccc}
(6, 3, 1, 4), & \boxed{(20, 5, 1, 18)},& (32, 13, 5, 30),& \boxed{(42, 7, 1, 40)},& (72, 9, 1, 70),\\
 (96, 43, 19, 94),& (102, 23, 5, 100), & (110, 11, 1, 108), & (122, 37, 11, 120),& (146, 53, 19, 144),\\
  \boxed{(150, 41, 11, 148)},& \boxed{(156, 13, 1, 154)}, &\boxed{(180, 75, 31, 178)},& (192, 89, 41, 190). &\\
\end{array}\]}
\end{exa}
For the last set of necessary conditions, we appeal to character theory. A \textit{character} $\chi$ is a homomorphism from a group $G$ into the complex numbers $\mathbb{C}$ under multiplication. In particular, the character $\chi_0: G \rightarrow \mathbb{C}$ defined by $\chi_0(g) = 1$ for all $g \in G$ is called the \textit{principal character}.

\noindent Now we can naturally extend a character of the group $G$ to a character of the corresponding group ring $\mathbb{Z}[G]$ as follows:

\begin{thm}
Let $\chi$ be a character of the group $G$. Then
$$
\chi \bigg(\sum_{g \in G} a_g X^{g} \bigg) := \sum_{g \in G}a_g \chi(g)
$$
is a homomorphism from $\mathbb{Z}[G]$ to $\mathbb{C}$.
\end{thm}

\noindent Next we state a useful property of group characters.
\begin{lem} \label{lem24}
If $\chi$ is a non-principal character of a group $G$, then
$$
\chi(G) = \sum_{g \in G} \chi(g) = 0
$$
\end{lem}

Now let $\chi$ be a character of $G = \mathbb{Z}_v$ and suppose that $D$ is an $(v, k, \lambda, t)$-ADS in $G$. Then by Proposition \ref{prop4}, in the group ring $\mathbb{Z}[G]$,
$$
D(X)D(X^{-1}) = (k-\lambda-1)\cdot \underline{1} +(\lambda + 1)G(X) - S(X)
$$
for some $t$-subset $S$ of $G$.

\noindent Now allowing $\chi$ to act on both sides of the above equation, we have
\begin{eqnarray*}
\chi(D(X)D(X^{-1}))  &=& \chi((k-\lambda-1)\cdot \underline{1} +(\lambda + 1)G(X) - S(X)) \\
&=& (k-\lambda-1)+(\lambda+1) \chi(G(X))-\chi(S(X))
\end{eqnarray*}

\noindent When $\chi = \chi_0$, the above relation yields
$$
k^2 = (k-\lambda-1)+(\lambda+1)v - t
$$
which is equivalent to the condition given in Theorem \ref{thm13}.

\noindent When $\chi \neq \chi_0$, we have
$$
|\chi(D(X))|^2 = (k-\lambda-1) - \chi(S(X))
$$
since $\chi$ is a homomorphism so $\chi(D(X^{-1})) = \overline{\chi(D(X))}$ and $\chi(G(X)) = 0$ by Lemma \ref{lem24}.

\noindent These last theorems are obtained by considering specific characters in the relation above.

\begin{thm}\cite{ZLZ}
If $2 \vert v$ and a $(v, k, \lambda, t)$-ADS exists in $\mathbb{Z}_v$, then
\begin{enumerate}
\item When $2\vert t$, there exists at least one square in the set
$$\bigg \{k-\lambda-(t+1-4l) \hs | \hs 0 \leq l \leq \frac{t}{2}\bigg \}$$
\item When $2\nmid t$ and $4 \vert v$, there exists at least one square in the set  $$ \bigg\{k-\lambda-(t+1-4l) \hs | \hs 0 \leq l \leq \frac{t-1}{2} \bigg\}$$
\item When $2 \nmid t$ and $4\nmid v$, there exists at least one square in the set  $$\bigg\{k-\lambda-(t-1-4l) \hs | \hs 0 \leq l \leq \frac{t-1}{2} \bigg\}$$
\end{enumerate}
\end{thm}

\begin{exa}{\rm
In $\mathbb{Z}_{v}$, there are no almost difference sets with parameters $(38,14,4, 3)$, $(68,25,8, 3)$, or $(86,23,5, 4)$.
\begin{proof}
We will argue the first set.

\noindent The parameter set $(38,14,4, 3)$ falls into the third category above, but the set
$$\bigg \{k-\lambda-(t-1-4l) \hs | \hs 0 \leq l \leq \frac{t}{2}\bigg \} = \{8,12\}
$$
is square free.
\end{proof}
}
\end{exa}

\noindent As a special case of the theorem above, we have the following corollary:
\begin{cor} \cite{ZLZ}
If $2 \vert v$ and there exists a $(v,k,\lambda, 1)$-ADS in $\mathbb{Z}_v$, then
\begin{enumerate}
\item If $4 \nmid v$, $k-\lambda$ is a square.
\item If $4 \vert v$, $k-\lambda-2$ is a square.
\item If $v \equiv 4(\text{mod } 8)$, $k-\lambda$ is the sum of two squares.
\end{enumerate}
\end{cor}

\noindent The next theorems come from considering the ternary character when $3 \vert v$.

\begin{thm}\cite{ZLZ}
When $3 \vert v$ and $2 \nmid t$, at least one element in the set
$$
\bigg \{ k - (\lambda +1)-(t-3l) \hs | \hs 0 \leq l \leq \frac{t-1}{2} \bigg \}
$$
has the form $x^2+y^2+xy$.
$\newline$
When $3 \vert v$ and $2 \vert t$, at least one element in the set
$$
\bigg \{k - (\lambda +1)-(t-3l) \hs | \hs 0 \leq l \leq \frac{t}{2} \bigg \} 
$$
has the form $x^2+y^2+xy$.
\end{thm}

\begin{thm}\cite{ZLZ}
When $3 \vert v$ and $2 \nmid (v-1-t)$, at least one element in the set
$$
\bigg\{k-\lambda+(v-1-t-3l) \hs | \hs 0 \leq l \leq \frac{v-t-2}{2}\bigg\}
$$
has the form $x^2+y^2+xy$.
$\newline$
When $3 \vert v$ and $2 \vert (v-1-t)$, at least one element in the set
$$
\bigg\{k-\lambda+(v-1-t-3l) \hs | \hs 0 \leq l \leq \frac{v-t-1}{2}\bigg\}
$$
has the form $x^2+y^2+xy$.
\end{thm}

\section{Difference Sets, Partial Difference Sets and Almost Difference Sets}

In this section, we highlight the interplay between almost difference sets, difference sets, and partial difference sets. We discuss how obtain difference sets from almost differences and contrariwise. We also discuss the circumstances under which a partial difference set qualifies as an almost difference set.

Let  $G$ be an additive group of order $v$ and let $D$ be a $k$-element subset of $G$. If the difference function $d_D(x)$ equals $\lambda$ for every nonzero element $x$ of $G$, we say that $D$ is a $(v, k, \lambda)$-difference set in $G$.  In first two subsections, we give a connection between certain difference sets and almost difference sets by demonstrating how to construct almost difference sets from difference sets and vice versa. The results below come from a paper by Arasu, Ding, Helleseth, Kumar, and Martinsen \cite{AD}. Throughout these first two subsections we let $G$ be an additive abelian group of order $v$ where $v \equiv 1 (\text{mod }4)$.

\subsection{Almost difference sets contained within difference sets}
$\newline$ $\newline$
\indent Our first lemma gives a parameter restriction for any ADS with $t = \frac{v-1}{2}$ obtained from a difference set by the removal of an element.

\begin{lem}\cite{AD} \label{lem31}
Let $D$ be a $(v, k, \lambda)$-difference set of $G$ and let $d \in D$. If $D \setminus \{d\}$ is an $(n,k-1, \lambda-1, \frac{v-1}{2})$-ADS in $G$, then
$$
k = \frac{v+3}{4} \hs \text{ and } \hs \lambda = \frac{v+3}{16}
$$
\end{lem}

\begin{proof}
By the necessary conditions of difference sets and almost difference sets, we have
$$
\begin{cases}
k(k-1) = \lambda v \\
(k-1)(k-2) = (\lambda-1)\frac{v-1}{2}+\lambda(\frac{v-1}{2}-1)
\end{cases}
$$
and solving this system of equations gives the desired conclusion.
\end{proof}

The following theorem shows how to construct an almost difference set by removing an element from a difference set.

\begin{thm}\cite{AD}
Let $D$ be a $(v, \frac{v+3}{4}, \frac{v+3}{16})$-difference set of $G$, and let $d$ be an element of $D$. If $2d$ cannot be written as the sum of two distinct elements of $D$, the $D \setminus \{d\}$ is $(v, \frac{v-1}{4}, \frac{v-13}{16}, \frac{v-1}{2})$-ADS of $G$.
\end{thm}

\begin{exa}{\rm
The set $\{0,1,3,9\}$ is a $(13,4,1)$-difference set in $\mathbb{Z}_{13}$ while the set $\{1,3,9\}$ is a $(13, 3, 1, 6)$-ADS in $\mathbb{Z}_{13}$. }
\end{exa}

The next theorem shows how to construct a difference set by adding an element to an almost difference set.

\begin{thm} \cite{AD}
Let $D = \{d_1, d_2, ..., d_{(v-1)/4}\}$ be a $(v, \frac{v-1}{4}, \frac{v-13}{16}, \frac{v-1}{2})$-ADS in $G$. Let $H$ be the subset of $G\setminus\{0\}$ such that $d_D(x) = \frac{v-13}{16}$ for each $x \in H$ and $d_D(x) = \frac{v+3}{16}$ for each nonzero $x \in G\setminus H$. Let $d \in G \setminus D$. If
\begin{enumerate}
\item $2d$ is not the sum of any two distinct elements of $D$, and
\item $d-d_i \in H$ and $d_i-d \in H$ for each $1 \leq i \leq \frac{v-1}{4}$.
\end{enumerate}
then $D \cup \{d\}$ is a $(v, \frac{v+3}{4}, \frac{v+3}{16})$-difference set of $G$.
\end{thm}

\begin{exa}{\rm
The set $D = \{0,4,6\}$ is $(13,3,1,6)$-ADS in $\mathbb{Z}_{13}$ and $H = \{1,3,5,8,10,12\}$ is the subset of $\mathbb{Z}_{13} \setminus\{0\}$ such that $d_D(x) = 0$ for each $x \in H$ and $d_D(x) = 1$ for each $x \in \mathbb{Z}_{13} \setminus (H \hs \cup \{0\}) = \{2, 4, 6, 7, 9, 11\}$. Further, $1 \in G \setminus D$ and satisfies the conditions of the above theorem. Thus, $ \{0, 1, 4, 6\}$ is a $(13, 4, 1)$- difference set in $\mathbb{Z}_{13}$.}
\end{exa}

\subsection{Difference sets contained within almost difference sets}
$\newline$ $\newline$
\indent Our first lemma gives a parameter restriction for any ADS with $t = \frac{v-1}{2}$ obtained from a DS by the addition of an element.

\begin{lem}\cite{AD}
Let $D$ be a $(v, k, \lambda)$-difference set of $G$ and let $d \in G \setminus D$. If $D \cup \{d\}$ is an $(n,k+1, \lambda, \frac{v-1}{2})$-ADS in $G$, then
$$
k = \frac{v-1}{4} \hs \text{ and } \hs \lambda = \frac{v-5}{16}
$$
\end{lem}

\begin{proof}
The proof is similar to Lemma \ref{lem31}.
\end{proof}

The following theorem shows how to construct an almost difference set by adding an element to difference set.

\begin{thm}\cite{AD} Let $D$ be a $(v, \frac{v-1}{4}, \frac{v-5}{16})$-difference set in $G$, and let $d \in G \setminus D$. If $2d$ cannot be written as the sum of two distinct elements of $D$, then $D \cup \{d\}$ is a $(v, \frac{v+3}{4}, \frac{v-5}{16}, \frac{v-1}{2})$-ADS in $G$.
\end{thm}

\begin{exa}{\rm
The set $D = \{0,1,4,14,16\}$ is a $(21, 5, 1)$-difference set in $\mathbb{Z}_{21}$ and $d = 3 \in \mathbb{Z}_{21} \setminus D$ such that $2d$ is not expressible as a sum of two distinct elements in $D$. Thus $D \hs \cup \hs \{d\} = \{0, 1, 3, 4, 14, 16\}$ is a $(21,6,1, 10)$-ADS in $\mathbb{Z}_{21}$.  }
\end{exa}

The next theorem shows how to construct a difference set by removing an element from an almost difference set.

\begin{thm} \cite{AD}
Let $D = \{d, d_1, d_2, ..., d_{(v-1)/4}\}$ be a $(v, \frac{v+3}{4}, \frac{v-5}{16}, \frac{v-1}{2})$-ADS in $G$. Let $H$ be the subset of $G\setminus\{0\}$ such that $d_D(x) = \frac{v-5}{16}$ for each $x \in H$ and $d_D(x) = \frac{v+11}{16}$ for each nonzero $x \in G\setminus H$. Let $d \in D$. If
\begin{enumerate}
\item $2d$ is not a sum of any two distinct elements of $D$, and
\item $d-d_i \in H$ and $d_i-d \in G\setminus(H \hs \cup \hs \{0\})$ for each $1 \leq i \leq \frac{v-1}{4}$.
\end{enumerate}
then $D \setminus \{d\}$ is a $(v, \frac{v-1}{4}, \frac{v-5}{16})$- difference set of $G$.
\end{thm}

\begin{exa}{\rm
The set $D = \{0,1,2,5,15,17\}$ is $(21,6,1,10)$-ADS in $\mathbb{Z}_{21}$ and $H = \{3, 7, 8, 9, 10, 11, 12, 13, 14, 18\}$ is the subset of $\mathbb{Z}_{21} \setminus\{0\}$ such that $d_D(x) = 1$ for each $x \in H$ and $d_D(x) = 2$ for each $x \in \mathbb{Z}_{21} \setminus (H \hs \cup \{0\}) = \{1, 2, 4, 5, 6, 15, 16, 17, 19, 20\}$. Further, $0 \in D$ and satisfies the conditions of the above theorem. Thus, $ \{1,2,5,15,17\}$ is a $(21, 5, 1)$- difference set in $\mathbb{Z}_{21}$.}
\end{exa}

\subsection{Almost difference sets which are partial difference sets}
$\newline$ $\newline$
\indent Almost difference sets also arise as a special type of partial difference sets. A $k$-element subset $D$ of a group $G$ is called a $(v, k, \lambda, \mu)$-\textit{partial difference set} (PDS) if the difference function $d_D(x)$ takes on the value $\lambda$ for each nonzero $x \in D$ and takes on the value $\mu$ for each nonzero $x \notin D$. That is, for each $x \in G \setminus \{0\}$,
$$
d_D(x) =
\begin{cases}
\lambda , & \text{for }x \in D \\
\mu , & \text{for } x \in G\setminus D
\end{cases}
$$
Thus, any PDS with $| \lambda - \mu| = 1$ is an ADS.  Furthermore, we say that $D$ is regular if in addition $0 \notin D$ and $D = -D$ where $-D := \{-d \hs | \hs d \in D\}$. Partial difference sets have also been extensively studied as they are related to Schur rings, two-weight codes, strongly regular Cayley graphs, and partial geometries. For more information on partial difference sets, there is a thorough treatment given by Ma \cite{Ma94}. As stated above, a $(v, k, \lambda, \mu)$-PDS with $\mu = \lambda+1$ is an ADS. For a regular PDS of this type, we have the following characterization:

\begin{thm}\cite{Ma94}
Let $G$ be an abelian group of order $v$ and let $D$ be a nontrivial regular $(v, k, \lambda, \mu)$- partial difference set with $\mu = \lambda+1$, then up to complementation, one of the following cases occurs:
\begin{itemize}
\item $(v, k, \lambda, \mu) = (v, \frac{v-1}{2}, \frac{v-5}{4}, \frac{v-1}{4})$ where $v \equiv 1(\text{mod }4)$
\item $(v, k, \lambda, \mu) = (243, 22,1,2)$
\end{itemize}
\end{thm}

A partial difference set of the first parameter type is also said to be Paley type since the Paley partial difference set, the set of nonzero squares in the finite field $GF(q)$ where $q \equiv 1 (\text{mod }4)$ is a prime power, has those parameters.  Furthermore, regular PDSs are quite common. In fact, if $D$ is a $(v,k, \lambda, \mu)$-PDS with $\lambda \neq \mu$, then $D = -D$ \cite{Ma94}. Thus, every ADS $D$ coming as a PDS with $0 \notin D$ and $\mu = \lambda+1$ is characterized by the theorem above.

\section{Almost Difference Sets from Cyclotomic Classes of Finite Fields} \label{sec7}
Many examples to come follow the classical approach of combining certain cyclotomic classes within a finite field. This technique has been used to construct difference sets (families) \cite{Wi} \cite{St}, partial difference sets \cite{Ma94}, and external difference families \cite{CD}, just to name a few. In this section, we provide the background for these cyclotomic constructions.

Throughout this section $p$ will denote an odd prime, $q = p^{\alpha}$ will denote an odd prime power, and $G = (GF(q),+)$ (resp. $(GF(p),+)$) will denote the additive group of the finite field $GF(q)$ (resp. $GF(p)$).

\subsection{Cyclotomy and the cyclotomic numbers}
$\newline$ $\newline$
\indent Let $q = p^{\alpha} = ef+1$ for positive integers $e$ and $f$. Then for a fixed primitive element $\gamma$ of $GF(q)$ we define
$$
C_0^e = \langle \gamma^e \rangle = \{\gamma^{es} \hs| \hs 0 \leq s \leq f-1\} \hspace{2mm} \text{ and } \hspace{2mm} C_i^e = \gamma^iC_0^e \hspace{2mm} \text{ for } i = 0, 1, ..., e-1
$$

These cosets are called the \textit{cyclotomic classes of order $e$} with respect to $GF(q)$. We further note that they partition the nonzero elements of the field. That is,
$$
GF(q)^{\times} = \bigcup\limits_{i = 0}^{e-1}  C_i^e
$$
where $GF(q)^{\times}$ is the multiplicative group of $GF(q)$.

\indent Next, for integers $0 \leq i, j \leq e-1$, we define
$$
(i,j)_e : = |(C_i^e+1) \cap C_j^e|
$$

These constants $(i,j)_e$ are called the \textit{cyclotomic numbers of order $e$} in $GF(q)$. There are at most $e^2$ distinct constants of order $e$. Additionally, these constants depend on $q$ and $\gamma$, a subtlety which is not reflected in the notation. Formulas for the cyclotomic numbers of order $e = 2, 3,$ and $4$ are provided in the Appendix.

Though there are many cyclotomic numbers of a given order, these constants obey quite a few symmetry properties which greatly reduce the number of distinct values. We state these properties in the following lemma.

\begin{lem}\cite{St}\label{lem8} For the cyclotomic numbers of order $e$ defined above, we have the following relations:

\begin{enumerate}
\item $(i,j)_e = (i', j')_e$ where $i \equiv i' (\text{mod } e)$ and $j \equiv j' (\text{mod }e)$

\item$ (i,j)_e = (e-i,j-i)_e =
\left\{
	\begin{array}{ll}
		(j,i)_e  & \mbox{if } f ~\mbox{is even} \\
		(j+\frac{e}{2},i+\frac{e}{2})_e & \mbox{if } f ~\mbox{is odd}
	\end{array}
\right.$
\item $\displaystyle \sum_{j=0}^{e-1}(i,j)_e=f-\theta_i$ where $ \theta_i=
\left\{
	\begin{array}{ll}
		1 & \mbox{if } f ~\mbox{is even and}~ i=0 \\
		1 & \mbox{if } f ~\mbox{is odd and}~ i=\frac{e}{2}\\
		0 &\mbox{otherwise}
	\end{array}
\right.$
\item $\displaystyle \sum_{i=0}^{e-1}(i,j)_e=f-\theta_j$ where $ \theta_j=
\left\{
	\begin{array}{ll}
		1 & j=0 \\
		0 &\mbox{otherwise}
	\end{array}
\right.$
 \item $\displaystyle \sum_{i=0}^{e-1}(i,i+j)_e=f-\theta_j$ where $ \theta_j=
\left\{
	\begin{array}{ll}
		1 & j=0 \\
		0 &\mbox{otherwise}
	\end{array}
\right.$
\end{enumerate}
\end{lem}

Now let $G$ denote the additive group of $GF(q)$ and consider the subset $D = C_i^e$. To determine whether $C_i^e$ qualifies as an ADS, by Proposition \ref{prop4}, it suffices to compute $C_i^e(X)C_i^e(X^{-1})$ in the group ring $\mathbb{Z}[G]$. Therefore, we need to determine $-C_i^e := \{ -x \hs | \hs x \in C_i^e \}$ since $C_i^e(X^{-1}) = -C_i^e(X)$. This set is described in the following lemma.

\begin{lem}\label{lem9}
$$
-C_i^e := \{-x \hs | \hs x \in C_i^e\} =
\begin{cases}
C_i^e \text{ if } f \text{ is even } \\
C_{i + \frac{e}{2}}^e \text{ if } f \text{ is odd }
\end{cases}
$$
\end{lem}

\begin{proof}
Since $q = ef+1$ is odd, it follows that $q-1 = ef$ is even. Furthermore,
$$
-1 = \gamma^{\frac{q-1}{2}} = \gamma^{\frac{ef}{2}} =
\begin{cases}
\gamma^{e\big(\frac{f}{2}\big)} \text{ if } f \text{ is even }\\
\gamma^{e\big(\frac{f-1}{2}\big) + \frac{e}{2}} \text{ if } f \text{ is odd }
\end{cases}
$$
Thus, we have that
$$
-C_i^e := \{-x \hs | \hs x \in C_i^e\} =
\begin{cases}
\left(\gamma^{e\big(\frac{f}{2}\big)}\right)C_i^e = C_i^e \text{ if } f \text{ is even } \\
\left(\gamma^{e\big(\frac{f-1}{2}\big) + \frac{e}{2}}\right)C_i^e = C_{i + \frac{e}{2}}^e \text{ if } f \text{ is odd }
\end{cases}
$$
\end{proof}

\indent We are now ready to draw the parallel between cyclotomic numbers and almost difference sets. Their relationship is established in the following lemma due to Storer \cite{St}.

\begin{lem} \label{lem10}
Let $G$ denote the additive group of the finite field $GF(q)$. Then in the group ring $\mathbb{Z}[G]$,
$$
C_i^e(X)C_j^e(X) = a_{ij}\cdot \underline{1} + \sum_{k = 0}^{e-1}(j-i, k-i)_eC_k^e(X)
$$
where
$$
a_{ij} =
\begin{cases}
f \text{ if } f \text{ is even and } j = i \\
f \text{ if } f \text{ is odd and } j = i+\frac{e}{2} \\
0 \text{ otherwise }
\end{cases}
$$
\end{lem}

\begin{proof}
The number of solutions to the equation
$$
\gamma^{es+i} + \gamma^{et+j} = \gamma^{er+k}
$$
is the number of solutions to the equation
$$
\gamma^{e(s-t)+(i-j)}+1 = \gamma^{e(r-t)+(k-j)}
$$
which is precisely the cyclotomic number $(i-j, k-j)_e$. Finally, by properties $(1)$ and $(2)$ of Lemma \ref{lem8}, we have that
$$
(i-j, k-j)_e \stackrel{(2)}{=} (e-(i-j), (k-j)-(i-j))_e \stackrel{(1)}{=} (j-i, k-i)_e
$$
Hence,
$$
C_i^e(X)C_j^e(X) = a_{ij}\cdot \underline{1} + \sum_{k = 0}^{e-1}(j-i, k-i)_eC_k^e(X)
$$
where
$$
a_{ij} =
\begin{cases}
f \text{ if }f \text{ is even and } j = i \\
f \text{ if }f \text{ is odd and } j = i+\frac{e}{2} \\
0 \text{ otherwise }
\end{cases}
$$
as desired.
 \end{proof}

Finally, combining Lemmas \ref{lem9} and \ref{lem10} above, we have the following corollaries which completely characterize the distribution of differences when $D$ is obtained as a union of cyclotomic classes.

\begin{cor}\label{cor38}
Let $G$ denote the additive group of the finite field $GF(q)$ and consider the subset $D = C_i^e$. Then in the group ring $\mathbb{Z}[G]$, we have
\begin{eqnarray*}
D(X)D(X^{-1}) &=& C_i^e(X)C_i^e(X^{-1}) \\
&=& C_i^e(X)\left(-C_i^e(X) \right) \\
&=&
\begin{cases}
C_i^e(X)C_i^e(X) = f\cdot \underline{1}+\sum\limits_{k=0}^{e-1}(0,k-i)_eC_k(X) \text{ if } f \text{ is even } \\ \\
C_i^e(X)C_{i+\frac{e}{2}}(X) = f \cdot \underline{1} + \sum\limits_{k = 0}^{e-1}(\frac{e}{2}, k-i) C_k(X) \text{ if } f \text{ is odd }
\end{cases}
\end{eqnarray*}
\end{cor}

\begin{cor}
Let $G$ denote the additive group of the finite field $GF(q)$ and consider the subset $D = \bigcup_{i \in \mc{I}} C_i^e$ for some $\mc{I} \subset \{0, 1, \dots, e-1\}$. Then in the group ring $\mathbb{Z}[G]$, we have

\begin{eqnarray*}
D(X)D(X^{-1}) &=& \sum\limits_{i \in \mc{I}} \sum\limits_{j \in \mc{I}} C_i^e(X)C_j^e(X^{-1}) \\
&=& \begin{cases}
f|\mc{I}|\cdot \underline{1} + \sum\limits_{k = 0}^{e-1}\left(\sum\limits_{i \in \mc{I}}\sum\limits_{j \in \mc{I}} (j-i,k-i)_e\right) C_k^e(X) \text{ if } f \text{ is even } \\ \\
f|\mc{I}| \cdot \underline{1} +\sum\limits_{k = 0}^{e-1}\left(\sum\limits_{i \in \mc{I}}\sum\limits_{j \in \mc{I}} (j+\frac{e}{2}-i,k-i)_e\right) C_k^e(X) \text{ if } f \text{ is odd }
\end{cases}
\end{eqnarray*}
\end{cor}

\subsection{Low order cyclotomic constructions}
$\newline$

Let $q = p^{\alpha} = ef+1$. In this section, we find all the almost difference sets in $G = (GF(q), +)$ of the form $D = \cup_{i \in \mc{I}} C_i^e$ or $D = \cup_{i \in \mc{I}}C_i^e \cup \{0\}$, where $\mc{I} \subset \{0, 1, \dots, e-1\}$, for orders $e = 2, 3$, and $4$. 

First note that, by Theorem \ref{comp} in Section \ref{nec}, we only need to consider $\mc{I}$ with $|\mc{I}| \leq \lfloor \frac{e}{2} \rfloor$. 

Next, we fully characterize when $C_i^e$, $C_i^e \cup \{0\}$, and $C_i \cup C_j$, for $0 \leq i \neq j \leq e-1$, provide almost difference sets. 

\begin{thm}\label{Ci}
Let $D = C_i^e$ for some $i \in \{0, 1, \dots e-1\}$. $D$ is an almost difference set in $G$ with parameters $(q, f, \lambda, t)$  if and only if 
$$
\begin{cases}
\left \{(0,j)_e \hs | \hs 0 \leq j \leq e-1\right \} = \{\lambda, \lambda +1\},  & \text{if } f \text{ is even} \\
 \{(\frac{e}{2},j)_e \hs | \hs 0 \leq j \leq e-1 \} = \{\lambda, \lambda +1\},  & \text{if } f \text{ is odd} 
\end{cases}
$$
\end{thm}
\begin{proof}
This is an immediate consequence of Corollary \ref{cor38}.
\end{proof}

\begin{thm} \label{Ci0}
Let $D = C_i^e \cup \{0\}$ for some $i \in \{0, 1, \dots e-1\}$. $D$ is an almost difference set in $G$ with parameters $(q, f+1, \lambda, t)$  if and only if 
$$
\begin{cases}
\{(0,0)_e+2\} \cup \{(0,j)_e \hs | \hs 1\leq j \leq e-1\} = \{\lambda, \lambda +1\},  & \text{if } f \text{ is even} \\
\{(\frac{e}{2},0)_e+1, (\frac{e}{2}, \frac{e}{2})_e+1\}\cup \{(\frac{e}{2},j)_e \hs | \hs 0 \leq j \leq e-1, j \neq 0, \frac{e}{2}\} = \{\lambda, \lambda +1\},  & \text{if } f \text{ is odd} 
\end{cases}
$$
\end{thm}
\begin{proof}
In the group ring $\mathbb{Z}[G]$, 
\begin{eqnarray*}
D(X)D(X^{-1}) &=& (C_i(X)+\underline{1})(C_i(X^{-1}+\underline{1}) \\
&=& C_i(X)C_i(X^{-1})+C_i(X)+C_i(X^{-1})+\underline{1} \\
&=&
\begin{cases}
 \left [f\cdot \underline{1}+\sum\limits_{k=0}^{e-1}(0,k-i)_eC_k(X)\right]+2C_i(X)+ \underline{1},& \text{ if } f \text{ is even } \\ 
\left [ f \cdot \underline{1} + \sum\limits_{k = 0}^{e-1}(\frac{e}{2}, k-i) C_k(X) \right]+C_i(X)+C_{\frac{e}{2}+i}(X) + \underline{1}, & \text{ if } f \text{ is odd }
 \end{cases}
\end{eqnarray*}
 by Corollary \ref{cor38}. The result now follows from Proposition \ref{prop4}.
 \end{proof}
 
 \begin{thm}\label{CiCj}
 Let $D = C_i^e \cup C_j^e$ for some $0 \leq i \neq j \leq e-1$. $D$ is an almost difference set in $G$ with parameters $(q, 2f, \lambda, t)$  if and only if 
$$
\begin{cases}
 \{(0,k-i)_e+(j-i,k-i)_e+(i-j,k-j)_e+(0,k-j)_e \hs | \hs 0\leq k\leq  e-1\} = \{\lambda, \lambda +1\},  & \text{if } f \text{ is even} \\
\{(\frac{e}{2}, k-i)_e+(j+\frac{e}{2}-i,k-i)_e +(i+\frac{e}{2}-j, k-j)_e+(\frac{e}{2}, k-j)_e \hs | \hs 0 \leq k \leq e-1\} = \{\lambda, \lambda +1\},  & \text{if } f \text{ is odd} 
\end{cases}
$$
 \end{thm}
 
  \begin{proof}
This is an immediate consequence of Corollary \ref{cor38}.
\end{proof} 
 
 \subsubsection{e = 2}\label{C2}
 $\newline$
 \indent First consider an odd prime power $q = 2f+1$. By the preliminary remarks above, we only have to consider the sets $C_0^2$ and $C_1^2$. 
 
 By Theorem \ref{Ci} and the cyclotomic numbers of order 2 given Appendix \ref{app1},  we have that $C_i^2$ is an almost difference set in $(GF(q),+)$ if and only if 
$$
\begin{cases}
\big \{\frac{f-2}{2},\frac{f}{2} \big \} = \{\lambda, \lambda +1\},  & \text{if } f \text{ is even} \\
 \big \{\frac{f-1}{2} \big \} = \{\lambda, \lambda +1\},  & \text{if } f \text{ is odd} 
 \end{cases}
$$
for some nonnegative integer $\lambda$. Thus, $C_i^2$ is an almost difference set if and only if $q \equiv1(\text{mod }4)$.
 
 \subsubsection{e = 3}
$\newline$
 \indent Now consider an odd prime power $q = 3f+1$. We note that, in this case, $f$ must be even. Now we are left to consider sets of the form $C_i^3$ and $C_i^3 \cup \{0\}$. 

By Theorem \ref{Ci} and the cyclotomic numbers of order $3$ given in Appendix \ref{app2}, we have that $C_i^3$ is an almost difference set in $(GF(q),+)$ if and only if 
$$
\bigg\{ \frac{1}{9}(c+q -8), -\frac{1}{18}(c +9d-2q+4), -\frac{1}{18}(c-9d-2q+4) \bigg\} = \{\lambda, \lambda +1\}
$$
for some nonnegative integer $\lambda$ where $4q = c^2+27d^2$ with $c \equiv 1(\text{mod }3)$, is the proper representation of $q$ if $p \equiv 1(\text{mod }3)$, where the sign of $d$ is ambiguously determined. Thus, $C_i^3$ is an almost difference set if and only if $q = 7, 19,$ or $25$. 

Next, by Theorem \ref{Ci0} and the cyclotomic numbers of order $3$ given in Appendix \ref{app2}, we have that $C_i^3 \cup \{0\}$ is an almost difference set in $(GF(q),+)$ if and only if 
$$
\bigg\{ \frac{1}{9}(c+q +10), -\frac{1}{18}(c +9d-2q+4), -\frac{1}{18}(c-9d-2q+4) \bigg\} = \{\lambda, \lambda +1\}
$$
for some nonnegative integer $\lambda$ where $4q = c^2+27d^2$ with $c \equiv 1(\text{mod }3)$. Thus, $C_i^3\cup \{0\}$ is an almost difference set if and only if $q = 13$ or $37$.

\subsubsection{e = 4}
$\newline$
\indent Lastly, consider an odd prime power $q = 4f+1$. In this case, we consider sets of the form $C_i^4$, $C_i^4 \cup \{0\}$, and $C_i \cup C_j$. 

By Theorem \ref{Ci} and the cyclotomic numbers of order $4$ given in Appendix \ref{app3}, we have that $C_i^4$ is an almost difference set in $(GF(q),+)$ if and only if 
$$
\begin{cases}
\{\frac{1}{16}(q+2s-7), \frac{1}{16}(q-2s-3) \} = \{\lambda, \lambda +1\},  & \text{if } f \text{ is even} \\
 \{\frac{1}{16}(q-6s-11), \frac{1}{16}(q+2s+8t-3), \frac{1}{16}(q+2s-3), \frac{1}{16}(q+2s-8t-3)\} = \{\lambda, \lambda +1\},  & \text{if } f \text{ is odd} 
 \end{cases}
 $$
for some nonnegative integer $\lambda$ where $q = s^2+4t^2$, with $s \equiv 1(\text{mod }4)$, is the proper representation of $q$ if $p \equiv 1(\text{mod }4)$, where the sign of $t$ is ambiguously determined. Thus, $C_i^4$ is an almost difference set if and only if $q \equiv 5 (\text{mod } 8)$ and $q = 25+4t^2$ or $9+4t^2$ or $q = 9$.

Next, by Theorem \ref{Ci0} and the cyclotomic numbers of order $4$ given in Appendix \ref{app3}, we have that $C_i^4 \cup \{0\}$ is an almost difference set in $(GF(q),+)$ if and only if 
$$
\begin{cases}
\{\frac{1}{16}(q+2s+9), \frac{1}{16}(q-2s-3) \} = \{\lambda, \lambda +1\},  & \text{if } f \text{ is even} \\
 \{\frac{1}{16}(q-6s+21), \frac{1}{16}(q+2s+8t-3), \frac{1}{16}(q+2s-3), \frac{1}{16}(q+2s-8t-3)\} = \{\lambda, \lambda +1\},  & \text{if } f \text{ is odd} 
 \end{cases}
 $$
for some nonnegative integer $\lambda$ where $q = s^2+4t^2$, with $s \equiv 1(\text{mod }4)$. Thus, $C_i^4 \cup \{0\}$ is an almost difference set if and only if $q \equiv 5(\text{mod }8)$ and $q = 1+4t^2$ or $49+4t^2$.

Lastly, we consider sets of the form $C_i^4 \cup C_j^4$ for some $0 \leq i \neq j \leq 3$. First note that $C_0^4 \cup C_2^4 = C_0^2$ and $C_1^4 \cup C_3^4 = C_1^2$ have already been considered. Thus we are left to consider sets of the form $C_i^4 \cup C_{i+1}^4$. By Theorem \ref{CiCj} and the cyclotomic numbers of order $4$ given in Appendix \ref{app3}, we have that $C_i^4 \cup C_{i+1}^4$ is an almost difference set in $(GF(q), +)$ if and only if
$$
\begin{cases}
 \{(0,k-i)_4+(1,k-i)_4+(3,k-i-1)_4+(0,k-i-1)_4 \hs | \hs 0\leq k\leq  e-1\} = \{\lambda, \lambda +1\},  & \text{if } f \text{ is even} \\
\{(2, k-i)_4+(3,k-i)_4 +(1, k-i-1)_4+(2, k-i-1)_4 \hs | \hs 0 \leq k \leq e-1\} = \{\lambda, \lambda +1\},  & \text{if } f \text{ is odd} 
 \end{cases}
 $$
 for some nonnegative integer $\lambda$ where $q = s^2+4t^2$, with $s \equiv 1(\text{mod }4)$. 
 
  Now for all $i \in \{0,1,2,3\}$, we have that
 \begin{eqnarray*}
 &&\{(0,k-i)_4+(1,k-i)_4+(3,k-i-1)_4+(0,k-i-1)_4 \hs | \hs 0\leq k\leq  e-1\}  \\
 &=& \bigg\{\frac14(q-2t-1), \frac14(q+2t-5),\frac14(q-2t-5), \frac14(q+2t-3)\bigg\} 
 \end{eqnarray*}
 and 
\begin{eqnarray*}
&& \{(2, k-i)_4+(3,k-i)_4 +(1, k-i-1)_4+(2, k-i-1)_4 \hs | \hs 0 \leq k \leq e-1\} \\
&=&\bigg \{\frac14(q-2t-3), \frac14(q+2t-3), \frac14(q-2t-3), \frac14(q+2t-3)\bigg \}
 \end{eqnarray*}
 
 Thus, $C_i^4 \cup C_{i+1}^4$ is an almost difference set if and only if $q \equiv 5 (\text{mod } 8)$ and $q = s^2+4$ with $s \equiv 1(\text{mod }4)$. 
 
 \subsubsection{Summary of Results}
 $\newline$
 \indent In conclusion, for $q = p^{\alpha} = ef+1$, the almost difference sets in $G = (GF(q), +)$ of the form $D = \cup_{i \in \mc{I}} C_i^e$ or $D = \cup_{i \in \mc{I}}C_i^e \cup \{0\}$, where $\mc{I} \subset \{0, 1, \dots, e-1\}$, for orders $e = 2, 3$, and $4$ are precisely:
 \begin{center}
 \begin{tabular}{| c | c |}
\hline
ADS &  Conditions  \\
\hline
$C_i^{2}$ & $q \equiv 1(\text{mod }4)$  \\
\hline
$C_i^3$ & $q = 7,19, 25$ \\
\hline
$C_i^3 \cup \{0\}$ & $q = 13, 37$ \\
\hline
$C_i^4$ & $q \equiv 5(\text{mod }8)$ and $q = 25+4t^2$ or $9+4t^2$, or $q = 9$\\
\hline
$C_i^4 \cup \{0\}$ & $q \equiv 5(\text{mod }8)$ and $q = 1+4t^2$ or $49+4t^2$ \\
\hline
$C_i^4 \cup C_{i+1}^4$ & $q \equiv 5(\text{mod }8)$ and $q = s^2+4$ with $s \equiv 1(\text{mod }4)$\\
\hline
\end{tabular}
\end{center}
up to complementation.

\subsection{Cyclotomic almost difference sets coming as partial difference sets}
$\newline$

As stated earlier, any $(v, k, \lambda, \mu)$-PDS with $|\lambda-\mu| = 1$ is a $(v, k, \lambda, t)$-ADS. Here we present two well known families of cyclotomic partial difference sets with $|\lambda - \mu| = 1$.

\begin{thm}(Payley, Todd) For $q \equiv 1 (\text{mod }4)$, $C_0^2$, the set of nonzero squares in $GF(q)$, is a $(q, \frac{q-1}{2}, \frac{q-5}{4}, \frac{q-1}{4})$-PDS in $(GF(q), +)$.
\end{thm}

\begin{proof}
Given in Section \ref{C2}.
\end{proof}

The next family of ADSs which are also PDSs comes as a special case of a theorem due to Calderbank and Kantor \cite{CK}.

\begin{thm}
For any set $\mc{I} \subset \{0, 1, ..., q\}$ with $|\mc{I}| = \frac{q\pm1}{2}$, $D = \bigcup_{i \in I}C_i^{q+1}$ is a \newline
$(q^2, \frac{q\pm 1}{2}(q-1), \frac{q\pm 1}{2}(\frac{q\pm1}{2}-1)-1, \frac{q\pm1}{2}(\frac{q\pm1}{2}-1))$-PDS in $(GF(q^2), +)$.
\end{thm}

\begin{proof}
We'll prove the case for $|\mc{I}| = \frac{q+1}{2}$.

Let $\mc{I} \subset \{0, 1, ..., q\}$ with $|\mc{I}| = \frac{q+1}{2}$ and consider $D = \bigcup_{i \in \mc{I}} C_i^{q+1}$.
By Proposition \ref{prop5}, to show that $D$ is a $(q^2, \frac{q^2-1}{2}, \frac{q^2-3}{2}, \frac{q^2-1}{2})$-PDS in $G = (GF(q^2),+)$, it suffices to show that in the group ring $\mathbb{Z}[G]$,
$$
D(X)D(X^{-1}) = \frac{q^2-1}{2}\cdot \underline{1} + \frac{q^2-3}{2} D(X) + \frac{q^2-1}{2} (G(X)-D(X)-\underline{1})
$$
Fix a primitive element $\gamma$ of $GF(q^2)$. Then
$$
GF(q) = \{x \in GF(q^2) \hs | \hs x^{q-1}=1\} \hs \cup \hs \{0\} = \{\gamma^{(q+1)s} \hs | \hs 0 \leq s \leq q-2\} \hs \cup \hs \{0\}
$$
is a subfield of $GF(q^2)$ and can be represented in terms of cyclotomic cosets as $GF(q) = C_0^{q+1} \cup \{0\}$.

Now if we view $GF(q^2)$ as a 2-dimensional vector space over $GF(q)$, the 1-dimensional subspaces are
$$
S_i = \gamma^i GF(q) = C_i^{q+1} \hs \cup \hs \{0\} \hs \text{ for }\hs 0 \leq i \leq q
$$
That is, the cyclotomic classes of order $q+1$ can be viewed as $1$-dimensional subspaces where the zero element has been removed.

Further, since the $S_i$ are 1-dimensional subspaces and $GF(q^2)$ is a 2-dimensional vector space over $GF(q)$, in the group ring $\mathbb{Z}[G]$,
$$
S_i(X)S_j(X)=
\begin{cases}
qS_i(X), & \text{if } i = j \\
G(X), & \text{if } i \neq j
\end{cases}
$$

Now,
\begin{eqnarray*}
D(X)D(X^{-1}) &=& \bigg(\sum\limits_{i \in \mc{I}}C_i^{q+1}(X)\bigg)\bigg(\sum\limits_{i \in \mc{I}}C_i^{q+1}(X^{-1})\bigg) \\
&=& \bigg(\sum\limits_{i \in \mc{I}}C_i^{q+1}(X)\bigg)^2 \hs \text{ since } f = q+1 \hs \text{ is even} \\
&=& \bigg( \sum\limits_{i \in \mc{I}} (S_i(X)-\underline{1})\bigg)^2 \\
&=& \sum\limits_{i,j}S_i(X)S_j(X) - (q+1)\sum\limits_{i \in \mc{I}} S_i(X) + \bigg(\frac{q+1}{2}\bigg)^2 \cdot \underline{1} \\
&=&  \sum\limits_{i \neq j }G(X) - \sum\limits_{i \in \mc{I}} S_i(X) +\bigg(\frac{q+1}{2}\bigg)^2 \cdot \underline{1} \\
&=& \frac{q^2-1}{4}G(X) -\bigg(\sum\limits_{i \in \mc{I}} C_i^{q+1}(X)+\frac{q+1}{2}\cdot \underline{1}\bigg)  + \bigg(\frac{q+1}{2}\bigg)^2 \cdot \underline{1} \\
 &=& \frac{q^2-1}{4}G(X) - D(X) + \frac{q^2-1}{4} \cdot \underline{1} \\
&=& \frac{q^2-1}{2} \cdot \underline{1} + \frac{q^2-3}{2} D(X) +\frac{q^2-1}{2}(G(X)-D(X)-\underline{1})
\end{eqnarray*}
\end{proof}

\subsection{Cyclotomic almost difference sets not coming as other difference sets}
$\newline$ $\newline$
We now present the known cyclotomic constructions which do not overlap with other types of difference sets.

\begin{thm} The known families of cyclotomic almost difference sets not coming as other difference sets are the following:
\begin{enumerate}
\item $C_0^{4}$ the set of quartic residues is a $(q, \frac{q-1}{4}, \frac{q-13}{16}, \frac{q-1}{2})$-ADS in $(GF(q),+)$ if $q = 25+4y^2$ or $q = 9+4y^2$ $($see Ding \cite{DThesis} and Cusick et al., \cite{CDR}$)$ \\
\item $C_0^4 \hs \cup \hs \{0\}$ the set of quartic residues together with $0$ is a $(q, \frac{q+3}{4}, \frac{q-5}{16}, \frac{q-1}{2})$-ADS in $(GF(q),+)$ if $q = 1+4y^2$  or $q = 49+4y^2$ $($see Ding, Helleseth, and Lam \cite{DHL}$)$ \\
\item $D_0^8$ the set of octic residues is a $(q, \frac{q-1}{8}, \frac{q-41}{64}, \frac{q-1}{2})$-ADS in $(GF(q),+)$ where $q \equiv 41 (\text{mod }64)$ and $q = 19^2+4y^2 = 1+2b^2$ for some integers $y$ and $b$ or $q \equiv 41(\text{mod }64)$ and $q = 13^2+4y^2 = 1+2b^2$ for some integers $y$ and $b$. $($see Ding \cite{DThesis} and Cusick et al., \cite{CDR}$)$ \\
\item $C_i^4 \hs \cup \hs C_{i+1}^4$ for any $i$ $($$0 \leq i \leq 3$$)$ is a $(q, \frac{q-1}{2}, \frac{q-5}{4}, \frac{q-1}{2})$-ADS in $(GF(q),+)$ if $q = x^2+4$ with $x \equiv 1(\text{mod }4)$ $($see Ding, Helleseth, and Lam \cite{DHL}$)$. \\
\item $C_2^4 \hs \cup \hs \{0\} = -(C_0^4 \hs \cup \hs \{0\})$ is a $(q, \frac{q+3}{4}, \frac{q-5}{16}, \frac{q-1}{2})$-ADS in $(GF(q),+)$ if $q = 1+4y^2$ or $q = 49+4y^2$ for some odd integer $y$. $($see X. Wang and J. Wang \cite{WW}$)$. \\
\item $C_0^8\hs \cup \hs C_1^8 \hs \cup \hs C_2^8 \hs \cup \hs C_5^8$ is a $(q, \frac{q-1}{2}, \frac{q-5}{4}, \frac{q-1}{2})$-ADS in $(GF(q),+)$ if $q = l^2$ where $l$ is a prime power of the form $l = t^2+2 \equiv 3(\text{mod }8)$. $($see Ding, Pott, and Wang \cite{DPW}$)$ \\
\item $D_0^8 \hs \cup \hs \{0\}$ is a $\left(p, \frac{p+7}{8}, \frac{p-9}{64}, \frac{3(p-1)}{4}\right)$-ADS in $(GF(p),+)$ if $p = 9+64y^2 = 1+8b^2$ for some odd integers $y$ and $b$. (this paper)
\end{enumerate}
\end{thm}

\noindent We will now verify construction $7$ to illustrate the proof technique.

\begin{proof}
Assume $p = 8f+1$ and $p = 9+64y^2 = 1+8b^2$ for some odd integers $y$ and $b$. Then we have the following result due to Lehmer \cite{Le1},

\begin{prop}
If $p = 8f+1$, then $C_0^8$ forms a $(p, \frac{p-1}{8}, \frac{p-9}{64})$-difference set  in $(GF(p), +)$ if and only if $p = 9+64y^2= 1+8b^2$ for some odd integers $y$ and $b$.
\end{prop}

Thus, $C_0^8$ is a difference set in $(GF(p), +)$ and by Proposition \ref{prop6}, in the group ring $\mathbb{Z}[G]$ we have
$$
C_0^8(X)C_0^8(X^{-1}) = \frac{p-1}{8}\cdot \underline{1} + \frac{p-9}{64}(G(X)-\underline{1})
$$

Now consider $D = C_0^8 \hs \cup \hs \{0\}$. In the group $\mathbb{Z}[G]$, we have

\begin{eqnarray*}
D(X)D(X^{-1}) &=& (C_0^8(X)+\underline{1})(C_0^8(X^{-1})+\underline{1}) \\
&=& C_0^8(X)C_0^8(X^{-1}) + C_0^8(X)+C_0^8(X^{-1}) +\underline{1} \\
&=& \bigg(\frac{p-1}{8} \cdot \underline{1} + \frac{p-9}{64}(G(X)-\underline{1})\bigg)+C_0^8(X)+C_0^8(X^{-1}) +\underline{1} \\
&=& \frac{p+7}{8}\cdot \underline{1} + \frac{p-9}{64}(G(X)-\underline{1}) + C_0^8(X) + C_4^8(X)  \hs \text{ since } f = \frac{p-1}{8} = b^2 \hs \text{ is odd}\\
&=& \frac{p+7}{8} \cdot \underline{1} + \frac{p-9}{64}(G(X)- (C_0^8 \hs \cup \hs C_4^8)(X)-\underline{1})+ \frac{p+55}{64} (C_0^8 \hs \cup \hs C_4^8)(X)
\end{eqnarray*}

Thus, by Proposition \ref{prop4}, $D = C_0^8 \hs \cup \hs \{0\}$ is a $\left(p, \frac{p+7}{8}, \frac{p-9}{64}, \frac{3(p-1)}{4}\right)$-ADS in $(GF(p), +)$.
\end{proof}

\section{Almost Difference Sets from Functions}

In this section we present some almost difference sets which arise as the support of certain functions.

The cyclotomic constructions in Section \ref{sec7} all reside within groups of odd order. However, there are several classes of almost difference sets which occur in abelian groups of even order.  We now present one of them.

\begin{thm}\cite{AD}
Let $q$ be an odd prime power and fix a primitive element $\gamma$ for $GF(q)$. Now define the function $f$ from $(\mathbb{Z}_{q-1}, +)$ to $(GF(2),+)$ by
$$
f(i) =
\begin{cases}
1, & \text{if } \gamma^i \in (C_1^2 -1) \\
0, & \text{otherwise }
\end{cases}
$$
and let $D = \{x \in \mathbb{Z}_{q-1} \hs | \hs f(x) =1 \}$ be the support of $f$. Then $D$ is a $(q-1, \frac{q-1}{2}, \frac{q-3}{4}, \frac{3q-5}{4})$-ADS in $\mathbb{Z}_{q-1}$ if $q \equiv 3 (\text{mod }4)$, and a $(q-1, \frac{q-1}{2}, \frac{q-5}{4}, \frac{q-1}{4})$-ADS in $\mathbb{Z}_{q-1}$ if $q \equiv 1 (\text{mod }4)$.\end{thm}

Next, we introduce some almost difference sets obtained from sequences with perfect nonlinearity.

Let $(A,+)$ and $(B,+)$ be abelian groups of order $m$ and $n$, respectively, and let $f$ be a function from $A$ to $B$. One measure of nonlinearity for $f$ is given by
$$
P_f := \max_{a \in A\setminus\{0\}} \max_{b \in B} Pr(f(x+a)-f(x) = b)
$$
where Pr($E$) denotes the probability that event $E$ occurs. Many are interested in functions with high nonlinearity. These functions have important applications in cryptography and coding theory \cite{CDR}.
$P_f$ and the nonlinearity of $f$ are inversely proportional. That is, the smaller the value of $P_f$, the higher the nonlinearity of $f$. Thus, for applications in coding theory and cryptography we wish to find functions with the smallest $P_f$ possible.  $P_f$ is bounded below by $\frac{1}{|B|}$ (see \cite{AD}) and this bound is attainable. When $P_f = \frac{1}{|B|}$, we say that $f$ has \textit{perfect nonlinearity}. The next theorem lists some functions with perfect nonlinearity.

\begin{lem}\cite{TS} \label{lem48}
The power function $x^s$ from $GF(p^m)$ to $GF(p^m)$, where $p$ is odd, has perfect nonlinearity $P_f = \frac{1}{p^m}$ for the following $s$:
\begin{itemize}
\item s = 2
\item $s = p^k+1$, where $\frac{m}{\gcd(m,k)}$ is odd
\item $s =  \frac{3^k+1}{2}$ where $p = 3$, $k$ is odd, and $\gcd(m,k) = 1$.
\end{itemize}
\end{lem}

Now we present a construct due to Arasu et al. \cite{AD} which uses functions with perfect nonlinearity to construct almost difference sets.

\begin{thm}\cite{AD}
Let $f$ be a function from an abelian group $(A, +)$ of order $n$ to another abelian group $(B,+)$ order order $n$ with perfect nonlinearity $P_f = \frac{1}{n}$ and define
$$
C_b = \{x \in A \hs | \hs f(x) = b\}
$$
and
$$
C = \bigcup_{b \in B} \left(\{b\} \times C_b\right) \hs \subset \hs B \times A
$$
Then $C$ is an $(n^2, n , 0, n-1)$-ADS in $B \times A$.
\end{thm}

Now combining the last result with the functions in Lemma \ref{lem48}, we have the following:

\begin{thm} \cite{AD}
Let $f(x) = x^s$ be a function from $GF(p^m)$ to $GF(p^m)$, where $p$ is odd. Define $C_b = \{x \in GF(p^m) \hs | \hs f(x) = b \}$ for each $b \in GF(p^m)$ and
$$
C = \bigcup_{b \in GF(p^m)} \left(\{b\} \times C_b \right) \hs \subset \hs GF(p^m) \times GF(p^m)
$$
If
\begin{itemize}
\item s = 2
\item $s = p^k+1$, where $\frac{m}{\gcd(m,k)}$ is odd
\item $s =  \frac{3^k+1}{2}$ where $p = 3$, $k$ is odd, and $\gcd(m,k) = 1$.
\end{itemize}
Then $C$ is a $(p^{2m}, p^m, 0, p^m-1)$-ADS in $GF(p^m) \times GF(p^m)$.
\end{thm}

\section{Almost Difference Sets from Binary Sequences with Three-Level Autocorrelation}\label{sec9}

We search for almost difference sets as a means of producing binary sequences with three-level autocorrelation. In some cases though, the sequences were identified first. In this section, we provide some examples of almost difference sets coming as characteristic sets of known binary sequences with three-level autocorrelation. 

Let $\{s(t)\}$ be a binary sequence with period $n$ and let $D$ be its support with $|D| = k$. Then, as stated earlier, if $\{s(t)\}$ has three-level autocorrelation and correlation values 
$$
C_s(w) =
\begin{cases}
n, & \text{for } w = 0 \\
n-4(k-\lambda), & \text{for } t \text{ nonzero elements in } \mathbb{Z}_n \\
n-4(k-\lambda-1), &\text{for the remaining } n-t-1 \text{ nonzero elements in } \mathbb{Z}_n
\end{cases}
$$
then $D$ in an $(n,k,\lambda, t)$-almost difference set in $\mathbb{Z}_n$.

The first construction is due to Arasu et al. \cite{AD} and comes from interweaving four closely related sequences with ideal autocorrelation.

\begin{thm} \cite{AD} \label{thm50}
Let $\{s(t)\}$ be a binary sequence with period $l$ which has ideal autocorrelation. Define the matrix $M = (m_{i, j})$ as follows:
$$
M =
 \begin{pmatrix}
  s(0) & s(1) & \cdots & s(l-1) \\
  \bar{s}(0+\delta) & \bar{s}(1+\delta) & \cdots & \bar{s}(l-1+\delta) \\
  \bar{s}(0) & \bar{s}(1) & \cdots & \bar{s}(l-1) \\
     \bar{s}(0+\delta) & \bar{s}(1+\delta) & \cdots & \bar{s}(l-1+\delta)
 \end{pmatrix}
 $$
 where $\bar{s}(i) = 1+s(i)$ is the complement of $s(i)$, and $0\leq \delta \leq l-1$ is any fixed integer.
 Now define $\{u(t)\}$ to be the sequence of period $4l$ given by
 $$
 u(t) = m_{i,j} \hs \text{ where } \hs i \equiv t \pmod{4} \text{ and } j \equiv t \pmod{l}
 $$
 Then $\{u(t)\}$ has optimal autocorrelation, given by
 $$
C_u(w) =
\begin{cases}
4l, & \text{for } w = 0 \\
-4, & \text{for } 3l \text{ nonzero elements in } \mathbb{Z}_n \\
0, &\text{for the remaining } l-1 \text{ nonzero elements in } \mathbb{Z}_n
\end{cases}
$$
 \end{thm}

 \noindent Additionally, Arasu et al. \cite{AD} were able to explicitly identify the support of this sequence.

 \begin{thm} \label{thm51}
 Let $D$ be the support of the sequence $\{u(t)\}$ in Theorem \ref{thm50}, and let $C$ be the support of the underlying sequence $\{s(t)\}$. Then,
 \begin{eqnarray*}
 D &=& ((l+1)C (\text{mod }4l)) \hs \cup \hs ((l+1)(C-\delta)^* + 3l(\text{mod } 4l))\\
 &\cup& ((l+1)C^*+2l(\text{mod }4l)) \hs \cup \hs ((l+1)(C-\delta)^* + 3l(\text{mod }4l))
 \end{eqnarray*}
 where $C^*$ and $(C-\delta)^*$  denote the complement of $C$ and $C-\delta$ in $\mathbb{Z}_l$, respectively. Furthermore,
 $$
 \phi (D) = \{0\} \times C \hs \cup \hs \{1\} \times (C-\delta)^* \hs \cup \hs \{2\} \times C^* \hs \cup \hs \{3\} \times (C-\delta)^*
 $$
 where $\phi(x) := (x (\text{mod }4), x(\text{mod }l))$ is the isomorphism from $\mathbb{Z}_{4l}$ to $\mathbb{Z}_4 \times \mathbb{Z}_l$.
 \end{thm}

 Theorem \ref{thm50} beckons us to look at sequences with ideal autocorrelation. In the next lemma, we highlight the relationship between these sequences and certain difference sets.

 \begin{lem}
 A binary sequence $\{s(t)\}$ with period $l \equiv 3(\text{mod }4)$ has ideal autocorrelation if and only if its support $D = \{t \in \mathbb{Z}_l \hs | \hs s(t) = 1\}$ is a $(l, \frac{l-1}{2}, \frac{l-3}{4})$ or $(l, \frac{l+1}{2}, \frac{l+1}{4})$-difference set in $\mathbb{Z}_l$.
 \end{lem}

 \begin{proof}
 A sequence $\{s(t)\}$ with period $l$ is said to have ideal autocorrelation if $l \equiv 3(\text{mod }4)$ and $C_s(w) = -1$ for all nonzero $w \in \mathbb{Z}_l$. By Lemma \ref{lem2},
$$
C_s(w) = l-4(k-d_D(w))
$$
where $D$ is the support of $\{s(t)\}$ and $k = |D|$. Thus, $\{s(t)\}$ has ideal autocorrelation if and only if
$d_D(w) = k - \frac{l+1}{4}$ for all nonzero $w \in \mathbb{Z}_l$. That is, $D$ is a $(l, k, k-\frac{l+1}{4})$-difference set in $\mathbb{Z}_l$. Further, from the necessary condition of difference sets, we have that
$$
k(k-1) = \bigg(k-\frac{l+1}{4}\bigg)(l-1)
$$
and solving for $k$, we obtain $k = \frac{l\pm1}{2}$. Hence $\{s(t)\}$ has ideal autocorrelation if and only if its support is a $(l,\frac{l-1}{2}, \frac{l-3}{4})$ or $(l, \frac{l+1}{2}, \frac{l+1}{4})$-difference set in $\mathbb{Z}_l$.
 \end{proof}

 Difference sets with the parameters $(l, \frac{l-1}{2}, \frac{l-3}{4})$ or $(l, \frac{l+1}{2}, \frac{l+1}{2})$ are called \textit{Payley Hadamard difference sets}. Cyclic Payley-Hadamard difference sets include the following classes:
 \begin{enumerate}
 \item with parameters $(p, \frac{p-1}{2}, \frac{p-3}{4})$, where $p \equiv 3(\text{mod }4)$ is prime, and the difference set just consists of all the quadratic residues in $\mathbb{Z}_p$;
 \item with parameters $(2^t-1, 2^{t-1}-1, 2^{t-2}-1)$, for descriptions of difference sets with these parameters see \cite{Dill}, \cite{Dill1}, \cite{GMW}, \cite{Pott}, and \cite{Xi};
 \item with parameters $(l, \frac{l-1}{2}, \frac{l-3}{4})$, $l = p(p+2)$ and both $p$ and $p+2$ are primes. These are the twin-prime difference sets, and may be defined as $\{(g,h) \in \mathbb{Z}_p \times \mathbb{Z}_{p+2} \hs | \hs g, h \neq 0 \text{ and } \chi(g)\chi(h) =1 \} \hs \cup \hs \{(g,0) \hs | \hs g \in \mathbb{Z}_p\}$ where $\chi(x) = 1$ if $x$ is a nonzero square in the corresponding field, and $\chi(x) = -1$ otherwise (see \cite{JP}).
 \item with parameters $(p, \frac{p-1}{2}, \frac{p-3}{4})$, where $p$ is a prime of the form $p = 4s^2+27$. They are the cyclotomic difference sets $D = C_0^6 \hs \cup \hs C_1^6 \hs \cup \hs C_3^6$ (see \cite{Jung}).
 \end{enumerate}

 \begin{thm} \cite{AD}\label{thm52}
 The construction in Theorem \ref{thm50} gives the following families of binary sequences of period $n$ with optimal autocorrelation:
 \begin{enumerate}
 \item $n = 4p$ where $p \equiv 3(\text{mod }4)$ is any prime;
 \item $n = 4(2^t-1)$, for any integer $t \geq 1$;
 \item $n = 4p(p+2)$, where $p$ and $p+2$ are any twin primes;
 \item $n = 4p$ where $p = 4s^2+27$ is any prime.
 \end{enumerate}
 \end{thm}

 \begin{proof}
 The characteristic sequences of the difference sets above have ideal autocorrelation. With these base sequences, the construction in Theorem \ref{thm50} gives the four classes of binary sequences of period $n$ described above.
 \end{proof}

 Finally, since the above sequences have optimal autocorrelation, it follows that their corresponding supports are almost difference sets. The parameters of these sets are given in the following theorem.

\begin{thm} \cite{AD}
The support $D$ (laid out in Theorem \ref{thm51}) of the four classes of binary sequences with optimal autocorrelation in Theorem \ref{thm52}  gives four classes of cyclic almost difference sets with the following parameters:
\begin{enumerate}
\item $(4p, 2p-1, p-2, p-1)$, where $p \equiv 3(\text{mod }4)$ is any prime;
\item $(4(2^t-1), 2^{t+1}-3, 2^t-3, 2^t-2)$, for any integer $t \geq 1$;
\item $(4p(p+2), 2p(p+2)-1, p(p+2)-2, p(p+2)-1)$, where $p$ and $p+2$ are twin primes;
\item $(4p, 2p-1, p-2, p-1)$, where $p = 4s^2+27$ is any prime.
\end{enumerate}
\end{thm}

\section{Almost Difference Sets from Direct Product Constructions}

In this section, we highlight some almost difference sets which reside in the direct product of two abelian groups.

\noindent The first construction comes as a special case of the following result due to Jungnickel.

\begin{thm} \cite{Jung1}
Let $D_1$ be an ordinary $(v,k,\lambda)$-difference set in a group $A$, and let $D_2$ be a difference set with parameters $(4u^2, 2u^2-u, u^2-u)$ in a group $B$. Then
$$
D := (D_2 \times D_1^*) \hs \cup \hs (D_2^* \times D_1)
$$
is a divisible difference set in $B \times A$ relative to the subgroup $\{1\} \times A$, with parameters $(4u^2, v, 2u^2v+2ku-uv, \lambda_1, \lambda_2)$, where
\begin{eqnarray*}
\lambda_1 &=& (2u^2-u)(v-2k)+4u^2\lambda \\
\lambda_2 &=& u^2v-uv+2ku
\end{eqnarray*}
and $D_2^*$ denotes the complement of $D_2$.
\end{thm}

Note that a divisible difference set is an almost difference set if and only if $|\lambda_1 - \lambda_2| = 1$. That being said, we observe a special case of the above theorem.

\begin{cor}\cite{Jung1} \label{cor55}
Let $D_1$ be an ordinary $(l, \frac{l-1}{2}, \frac{l-3}{4})$ $($resp. $(l, \frac{l+1}{2}, \frac{l+1}{4})$$)$ difference set in $\mathbb{Z}_l$ and let $D_2$ be a trivial difference set in $\mathbb{Z}_4$ with parameters $(4,1,0)$. Then,
$$
D := (D_2 \times D_1^*) \hs \cup \hs (D_2^* \times D_1) = (\{i\} \times D_1^*) \hs \cup \hs (\{i+1, i+2, i+3\} \times D_1)
$$
 for any $i \in \mathbb{Z}_4$ is a $(4l, 2l-1, l-2, l-1)$ $($resp. $(4l, 2l+1, l, l-1)$$)$ almost difference set in $\mathbb{Z}_4 \times \mathbb{Z}_l$.
\end{cor}

By taking $D_1$ to be any of the Paley-Hadamard difference sets listed in Section \ref{sec9}, Corollary \ref{cor55} yields an almost difference  set and corresponding binary sequence with optimal autocorrelation.

Next we present several cyclotomic constructions.

The first two constructions are due to Ding, Helleseth, and Martinsen. They considered the finite field $GF(q)$ with $q \equiv 5(\text{mod }8)$. In this case, $q$ has a quadratic partition $q = x^2 + 4y^2$ with  $x \equiv \pm 1 (\text{mod } 4)$ \cite{St}.

\begin{thm}\cite{DH}
Let $i, j, l \in \{0, 1, 2, 3\}$ be three pairwise distinct integers, and define
$$
D = ( \{0\} \times (C_i^4 \hs \cup \hs C_j^4)) \hs \cup \hs (\{1\} \times (C_l^4 \hs \cup C_j^4))
$$
Then $D$ is a $(2q, q-1, \frac{q-3}{2}, \frac{3(q-1)}{2})$-ADS in $G = GF(2) \times GF(q)$ if
\begin{itemize}
\item $y = 1$ and $(i,j,k) = (0,1,3)$ or $(0,2,1)$, or
\item $x = 1$ and $(i,j,k) = (1,0,3)$ or $(0,1,2)$
\end{itemize}
\end{thm}

 \begin{thm}\cite{DH} \label{thm57}
Let $i, j, l \in \{0, 1, 2, 3\}$ be three pairwise distinct integers, and define
$$
D = ( \{0\} \times (C_i^4 \hs \cup \hs C_j^4)) \hs \cup \hs (\{1\} \times (C_l^4 \hs \cup C_j^4)) \hs \cup \hs \{(0,0)\}
$$
Then $D$ is a $(2q, q, \frac{q-1}{2}, \frac{3q-1}{2})$-ADS in $G = GF(2) \times GF(q)$ if
\begin{itemize}
\item $y = 1$ and $(i,j,k) \in \{(0,1,3), (0,2,3), (1,2,0), (1,3,0)\}$, or
\item $x = 1$ and $(i,j,k) \in \{(0,1,2), (0,3,2), (1,0,3), (1,2,3) \}$
\end{itemize}
\end{thm}

The next set of results is due to Zhang,  Lei, and Zhang. For the first construction, they considered a prime power $q = 2f+1$ where $f \equiv 1(\text{mod }2)$. In this case, the cyclotomic class $C_1^2$ is a $(q, \frac{q-1}{2}, \frac{q-3}{4})$-difference set in $(GF(q),+)$. Thus, by Corollary \ref{cor55}, we have that $D = (\{0\} \times (C_1^2)^*) \hs \cup \hs (\{1,2,3\} \times C_1^2)$ is a $(4q, 2q-1, q-2, q-1)$-ADS in $\mathbb{Z}_4 \times GF(q)$. The following construction is obtained by augmenting this construction with the elements $(1,0)$ and $(3,0)$.

\begin{thm}\cite{ZLZ}\label{thm58}
Let $q = 2f+1$ be a prime power with $f \equiv 1(\text{mod }2)$. Then
$$
D = (\{0\} \times C_0^2) \hs \cup \hs  (\{1,2,3\}) \times C_1^2) \hs \cup \hs \{(0,0), (1,0), (3,0)\}
$$
is a $(4q, 2q+1, q, q-1)$-ADS in $\mathbb{Z}_4 \times GF(q)$.
\end{thm}

Next Zhang, Lei, and Zhang set out to construct almost difference sets in $G = GF(p) \times GF(q)$, where $p$ and $q$ are prime powers. They obtained the following results:

\begin{thm} \cite{ZLZ}
Let $G = GF(p) \times GF(q)$. Then
$$
D := \{ (a,b) \in G \hs | \hs a,b \text{ are both square or nonsquare} \}
$$
is an almost difference set in $G$ if and only if $q = p \pm 2$ or $q = p$.
\end{thm}

\begin{rem}\cite{ZLZ}
When $q = p+2$, $D$ is a $(p(p+2), \frac{p^2-1}{2}, \frac{(p+1)(p-3)}{3}, p-1)$-cyclic almost difference set. In fact, $D$ is isomorphic to the almost difference set in $\mathbb{Z}_{p(p+2)}$ described in Ke's thesis \cite{Ke}. When $p = q$, $D$ is a $(p^2, \frac{(p-1)^2}{2}, \frac{p^2-4p+3}{4}, \frac{p^2+2p-3}{2})$-ADS which is not cyclic.
\end{rem}

\begin{thm} \cite{ZLZ}
Let $G= GF(p) \times GF(q)$ where $p \not \equiv q(\text{mod }4)$. Then,
$$
D = (GF(p) \times \{0\}) \hs \cup \hs \{(a,b) \in G \hs | \hs a,b \text{ are both square or nonsquare}\}
$$
cannot be an almost difference set and $D$ is a difference set of $G$ if and only if $q = p+2$. In fact, it is the twin-prime difference set \cite{JP}.
\end{thm}

The next construction, based on Paley-Hadamard difference sets, was recently given by Tang and Ding. It  elegantly generalizes three previously known constructions already covered above. We present it now to give the reader a sense of the historical time line surrounding almost difference sets.

\begin{thm}\cite{TD}
Suppose that $A$ and $B$ are, respectively, Paley-Hadamard difference sets with parameters $(l, \frac{l+1}{2}, \frac{l+1}{4})$ or $(l, \frac{l-1}{2}, \frac{l-3}{4})$ in $(G, +)$ and define
$$
D := (\{0,2\} \times A) \hs \cup \hs (\{1\} \times B) \hs \cup \hs (\{3\} \times B^*)
$$
where $B^*$ is the complement of $B$. Then $D$ is a $(4l, 2l+1, l, l-1)$ or $(4l, 2l-1, l-2, l-1)$-ADS in $(\mathbb{Z}_4, +) \times (G,+)$.
\end{thm}

\begin{rem}
The above theorem generalizes the constructions in Theorems \ref{thm51} and \ref{thm58} as well as Corollary \ref{cor55}. Let $A = B^*+\delta$ for some $0 \leq \delta < l$, then the construction above is the same as the one given in Theorem \ref{thm51}. Next, let $A = B$, then the construction above is the same as Corollary \ref{cor55}. Finally, we obtain the construction in Theorem \ref{thm58} by letting $A = C_1^2 \hs \cup \hs \{0\}$ and $B = C_1^2$ in the theorem above.
\end{rem}

The final construction presented here, due to Ding, Pott, and Wang, uses skew Hadamard difference sets.  Recall that a difference set $D$ of $(G,+)$ is called a \textit{skew Hadamard difference set} if $G = D \hs  \dot\cup \hs (-D) \hs  \dot\cup \hs \{0\}$.

\begin{thm}\cite{DPW}
Let $q$ and $q+4$ be two prime powers with $q \equiv 3(\text{mod }4)$. Let $E$ and $F$ be two skew Hadamard difference sets in the abelian groups $(GF(q),+)$ and $(GF(q+4),+)$ respectively. Then the set
$$
D := (E\times F) \hs \cup \hs (-E \times -F) \hs \cup \hs (GF(q) \times \{0\})
$$
is a $(q(q+4), \frac{q^2+4q-3}{2}, \frac{q^2+4q-9}{4}, \frac{q^2+4q-5}{2})$-ADS in $(GF(q),+) \times (GF(q+4),+)$.
\end{thm}

\begin{rem}\cite{DPW}
In addition to the classical Paley skew Hadamard difference sets, there are several other classes of skew Hadamard difference sets $($see \cite{DPW1}, \cite{DWX}, \cite{DY}$)$. Thus the above theorem can generate infinitely many almost difference sets.
\end{rem}

\section{Concluding Remarks and Open Problems}

Almost difference sets are of interest in combinatorics and find application in many areas. However, compared with difference sets, much less progress has been made on the construction of almost difference sets in the past ten years. In this section, we leave the reader with some open problems.

 \begin{enumerate}
 \item Find $(n, \frac{n-3}{2}, \frac{n-7}{4}, n-3)$-ADSs in $\mathbb{Z}_n$. These will generate sequences of period $n \equiv 3(\text{mod } 4)$ with optimal autocorrelation values $-1$ and $3$, where $k = \frac{n-3}{2}$. \\
 \item Find $(n, \frac{n-1}{2}, \frac{n-5}{4}, \frac{n-1}{2})$-ADSs in $\mathbb{Z}_n$. These will generate sequences of period $n \equiv 1 (\text{mod }4)$, where $k = \frac{n-1}{2}$.\\
 \item Find $(n, \frac{n-3}{2}, \frac{n-9}{4}, \frac{n-5}{2})$-ADSs in $\mathbb{Z}_n$. These will generate sequences of period $n \equiv 1(\text{mod }4)$ with optimal autocorrelation values $-1$ and $3$, where $k = \frac{n-3}{2}$. \\
 \item Find $(n, \frac{n}{2}-1, \frac{n-6}{4}, \frac{3(n-2)}{4})$ and $(n, \frac{n}{2}, \frac{n-2}{3}, \frac{3n-2}{4})$-ADSs in $\mathbb{Z}_n$. These will generate sequences of period $n \equiv 2 (\text{mod }4)$, where $k = \frac{n}{2}-1$ or $\frac{n}{2}$, respectively. \\
 \item In addition to the sequence given in section 9 and those given in \cite{LCE} and \cite{NCS}, construct other families of binary sequences of period $n \equiv 0(\text{mod }4)$.
 \end{enumerate}
\newpage

\appendix 
\section{Cyclotomic Numbers of Order 2} \label{app1}
Reference for the following formulas of cyclotomic numbers of order $2$ is \cite{St}.

When $e = 2$, the cyclotomic numbers are given by

\begin{enumerate}
\item $(0,0)_2 = \frac{f-2}{2}$; $(0,1)_2 = (1,0)_2 = (1,1)_2 = \frac{f}{2}$ if $f$ is even; and
\item $(0,0)_2 = (1,0)_2 = (1,1)_2 = \frac{f-1}{2}$; $(0,1)_2 = \frac{f+1}{2}$ if $f$ is odd
\end{enumerate}

\section{Cyclotomic Numbers of Order 3} \label{app2}
Note that for $e = 3$, $f$ must be odd. Reference for the following formulas of cyclotomic numbers of order $3$ is \cite{St}. 

The cyclotomic numbers for $e = 3$ are given by 

\begin{table}[!h]
\begin{center}
\label{Table 1}
\[\begin{array}{| c | c c c |}
\hline
(j,k) & 0 & 1 & 2  \\
\hline
0 & A & B & C  \\
1& B & C & D  \\
2 & C & D & B  \\
\hline
\end{array}\]
\end{center}
\end{table}

and the relations
\begin{eqnarray*}
9A &=& q-8+c \\
18B &=& 2q-4-c-9d \\
18C &=& 2q-4-c+9d \\
9D &=& q+1+c
\end{eqnarray*}
where $4q = c^2+27d^2$ with $c \equiv 1(\text{mod }3)$. 

Here it must be remarked that, if $q = p^{\alpha} \equiv 1(\text{mod }3)$, then $4q$ always admits at least one representation $c^2+27d^2$ (\cite{Dav}, Chapter 6), and, if $\alpha = 1$, at most one. For $\alpha > 1$, however, the representation is, in general, not unique. A representation $nq = x^2+Dy^2$ is said to be \textit{proper} if $(q,x) = 1$; if $p \equiv 1(\text{mod }3)$, then $q = p^{\alpha}$ has exactly one proper representation $4q = c^2+27d^2$, and it is this representation which gives rise to the cyclotomic number for $e = 3$. 

\section{Cyclotomic Numbers of Order 4} \label{app3}

Reference for the following formulas of cyclotomic numbers of order $4$ is \cite{St}. 

When $e =4$ and $f$ is odd, the cyclotomic number are given by 

\begin{table}[!h]
\begin{center}
\label{Table 1}
\[\begin{array}{| c | c c c c |}
\hline
(j,k) & 0 & 1 & 2 & 3 \\
\hline
0 & A & B & C & D \\
1& E & E & D & B  \\
2 & A & E & A & E \\
3 & E & D & B & E  \\
\hline
\end{array}\]
\end{center}
\end{table}

together with the relations

\begin{eqnarray*}
16A &=& q-7+2s\\
16B &=& q+1+2s-8t\\
16C &=& q+1-6s \\
16D &=& q+1+2s+8t \\
16E &=& q-3-2s
\end{eqnarray*}

where $q = s^2+4t^2$, with $s \equiv 1(\text{mod }4)$, is the proper representation of $q$ if $p \equiv 1(\text{mod }4)$, where the sign of $t$ is ambiguously determined. 

When $e =4$ and $f$ is even, the cyclotomic number are given by 

\begin{table}[!h]
\begin{center}
\label{Table 1}
\[\begin{array}{| c | c c c c |}
\hline
(j,k) & 0 & 1 & 2 & 3 \\
\hline
0 & A & B & C & D \\
1& B & D & E & E  \\
2 & C & E & C & E \\
3 & D & E & E & B  \\
\hline
\end{array}\]
\end{center}
\end{table}

together with the relations

\begin{eqnarray*}
16A &=& q-11-6s \\
16B &=& q-3+2s+8t\\
16C &=& q-3+2s \\
16D &=& q-3+2s-8t \\
16E &=& q+1-2s
\end{eqnarray*}

where $q = s^2+4t^2$, with $s \equiv 1(\text{mod }4)$, is the proper representation of $q$ if $p \equiv 1(\text{mod }4)$; the sign of $t$ is ambiguously determined.

\vskip 0.5in

\end{document}